\documentclass[11pt]{amsart}
\usepackage{amsthm,amssymb,amsfonts,amscd}
\usepackage{mathtools,mathrsfs}
\usepackage{microtype}
\usepackage{array}
\usepackage{enumitem}
\usepackage{xcolor}
\usepackage[colorlinks=true,
linkcolor=black,
urlcolor=blue,
citecolor=black]{hyperref}
\usepackage{cleveref}
\usepackage[margin=1in]{geometry}
\allowdisplaybreaks

\usepackage{silence}
\setlist[itemize]{leftmargin=1.8em,itemsep=2pt,topsep=4pt}
\setlist[enumerate]{leftmargin=2.2em,itemsep=2pt,topsep=4pt}

\theoremstyle{definition}
\newtheorem{thm}{Theorem}[section]
\newtheorem{cor}[thm]{Corollary}
\newtheorem{prop}[thm]{Proposition}
\newtheorem{lem}[thm]{Lemma}
\newtheorem{defn}[thm]{Definition}

\newtheorem{rmk}[thm]{Remark}

\newcommand{\mr}{\mathrm}
\newcommand{\bb}{\mathbb}
\newcommand{\mc}{\mathcal}
\newcommand{\mf}{\mathfrak}
\newcommand{\Z}{\bb Z}
\newcommand{\Q}{\bb Q}
\newcommand{\F}{\bb F}

\newcommand{\bP}{\bb P}
\newcommand{\bE}{\bb E}
\newcommand{\ve}{\varepsilon}
\newcommand{\wtil}{\widetilde}
\newcommand{\sq}{\mr{sq}}
\newcommand{\nsq}{\mr{nsq}}

\newcommand{\GL}{\mr{GL}}
\newcommand{\SL}{\mr{SL}}
\newcommand{\mO}{\mr{O}}
\newcommand{\SO}{\mr{SO}}
\newcommand{\M}{\mr M}
\newcommand{\cok}{\mr{cok}}
\newcommand{\rank}{\mr{rank}}
\newcommand{\im}{\mr{im}}
\newcommand{\Hom}{\mr{Hom}}
\newcommand{\Sur}{\mr{Sur}}
\newcommand{\Aut}{\mr{Aut}}
\newcommand{\End}{\mr{End}}
\newcommand{\Sym}{\mr{Sym}}
\newcommand{\Alt}{\mr{Alt}}
\newcommand{\GSp}{\mr{GSp}}
\newcommand{\Sp}{\mr{Sp}}
\newcommand{\Stab}{\mr{Stab}}
\newcommand{\Fix}{\mr{Fix}}
\newcommand{\wt}{\mr{wt}}
\newcommand{\tr}{\mr{tr}}

\newcommand{\tors}{\mr{tors}}
\newcommand{\dist}{\xrightarrow{\mr{dist}}}
\newcommand{\angles}[1]{\left\langle #1\right\rangle}

\title[Linearization of $p$-adic orthogonal matrices]
{Cokernels of random $p$-adic orthogonal matrices and their linearizations}
\author{Jungin Lee and Myungjun Yu}
\date{}
\address{J. Lee -- Department of Mathematics, Ajou University, Suwon 16499, Republic of Korea \newline M. Yu -- Department of Mathematics, Yonsei University, Seoul 03722, Republic of Korea}
\email{jileemath@ajou.ac.kr, mjyu@yonsei.ac.kr}

\begin{document}

\begin{abstract}
We study the linearization of the random $p$-adic orthogonal matrix model. Let $A_n$ and $B_n$ be Haar-random $n\times n$ special orthogonal and alternating matrices over $\mathbb{Z}_p$, respectively. For an odd prime $p$, we prove that $\mathrm{cok}(A_n-I_n)$ and $\mathrm{cok}(B_n)$ have the same limiting distribution as $n \to \infty$ through integers of the same parity. The two parity-dependent limits are interchanged for orthogonal matrices of determinant $-1$, while the full orthogonal group gives their equal mixture. For $p=2$, linearization does not preserve the limiting cokernel distribution, and we explicitly determine the limiting distributions for the special orthogonal model in terms of the spinor norm, which controls the parity of the $2$-adic valuation of the order of the torsion subgroup.
\end{abstract}

\maketitle

\vspace{-5mm}
\section{Introduction}\label{Sec1}

For a prime $p$, let $\Z_p$ and $\Q_p$ denote the ring of $p$-adic integers and the field of $p$-adic numbers, respectively. For a positive integer $m$, let $v_p(m)$ be the exponent of $p$ in $m$, and extend $v_p$ to $\Q^\times$ and $\Q_p^\times$ in the usual way. For a commutative ring $R$, write $\M_{m \times n}(R)$ for the set of $m \times n$ matrices over $R$, set $\M_n(R):=\M_{n \times n}(R)$, and write $\GL_n(R)$ for the set of invertible $n\times n$ matrices over $R$. Denote by $I_n$ the $n \times n$ identity matrix, and the transpose of a matrix $A$ by $A^T$. For $A \in \M_n(R)$, regard $A$ as an $R$-linear map $R^n\to R^n$ ($x \mapsto Ax$) and define the \emph{cokernel} of $A$ by $\cok(A) := R^n/\mr{im}(A)$.

\subsection{Linearization of random matrix models} \label{Sub11}

Cohen--Lenstra heuristics \cite{CL84} predict the distribution of the $\ell$-primary part of the class group of imaginary quadratic fields for an odd prime $\ell$. As a function-field analogue, Friedman and Washington \cite{FW89} studied the distribution of the $\ell$-primary part of the degree-zero Picard group $\mr{Pic}^0(C)$, where $C$ ranges over smooth projective hyperelliptic curves of genus $g$ over a finite field $\F_q$ and $\ell \nmid q$. Let $J_C$ be the Jacobian of $C$, so that $\mr{Pic}^0(C)\cong J_C(\F_q)$, and let $F$ denote the $q$-Frobenius action on the $\ell$-adic Tate module $T_\ell J_C\cong\Z_\ell^{2g}$. Then
\begin{equation} \label{eq1a}
\mr{Pic}^0(C)_\ell \cong \cok(I_{2g}-F).
\end{equation}
After choosing a $\Z_\ell$-basis of the Tate module, we may regard $F$ as an element of $\M_{2g}(\Z_\ell)$. In fact, with respect to the Weil pairing on this $\ell$-adic Tate module, $F$ is a symplectic similitude with multiplier $q$. 
Motivated by this description, Friedman and Washington considered $\cok(I_{2g}-L)$ for a Haar-random matrix $L\in\GL_{2g}(\Z_\ell)$. They also studied $\cok(M)$ for a Haar-random matrix $M\in \M_{2g}(\Z_\ell)$ and showed that the cokernel distributions in both models converge to the Cohen--Lenstra distribution as $g\to\infty$.

The limiting distribution in this arithmetic setting depends on $e=v_\ell(q-1)$. When $e=0$, Ellenberg, Venkatesh and Westerland \cite[Theorem 1.2]{EVW16} proved that the upper and lower densities of the distribution of $\mr{Pic}^0(C)_\ell$ (as $g\to\infty$) both converge to the Cohen--Lenstra distribution as $q \to \infty$.
When $e=1$ or $2$, Garton determined the limiting distribution of $\mr{Pic}^0(C_f)_\ell$, where $C_f$ is the hyperelliptic curve $y^2=f(x)$ and $f$ ranges over the monic squarefree polynomials of degree $2g+1$ in $\F_q[x]$, with $q\to\infty$ subject to $v_{\ell}(q-1)=e$ followed by $g\to\infty$ \cite[Theorem 1.2.4]{Gar15}. He also computed the moments of the corresponding random matrix distribution for every fixed $e \ge 1$ \cite[Corollary 3.2.7]{Gar15}.
Lipnowski, Sawin, and Tsimerman \cite[Theorem 1.1]{LST20} determined the limiting distribution of $\mr{Pic}^0(C)_\ell$ as $C$ ranges over smooth projective hyperelliptic curves of genus $g_i$ over $\F_{q_i}$, for any sequences $g_i,q_i\to\infty$ with each $q_i$ an odd prime power satisfying $v_\ell(q_i-1)=e$ for a fixed $e \ge 1$.

Subsequent work has extended these results in several directions. Sawin and Wood obtained $q\to\infty$ results for the distribution of class groups of general $\Gamma$-extensions in the presence of roots of unity \cite[Corollary 8.5]{SW26a}, and more recently studied maximal unramified extensions of $\Gamma$-extensions of $\F_q(t)$ \cite[Theorem 1.4]{SW26b}. For each fixed target group, Landesman and Levy computed Cohen--Lenstra moments for quadratic function fields over fixed sufficiently large $q$ \cite[Theorem 1.2.1]{LL24}, and extended this to nonabelian Cohen--Lenstra--Martinet moments for general $\Gamma$-extensions \cite[Theorem 1.3.2]{LL25}.

The nonlinear random matrix model studied in \cite{LST20} is given as follows. Write
$$
J_{2g}=\begin{pmatrix}0&I_g\\-I_g&0\end{pmatrix} \in \M_{2g}(\Z_\ell)
$$
and take a random matrix
$$
F\in\GSp_{2g}^{(q)}(\Z_\ell):=\{A\in\GL_{2g}(\Z_\ell) : A^T J_{2g}A=qJ_{2g}\},
$$
where $\GSp_{2g}^{(q)}(\Z_\ell)$ is equipped with the probability measure obtained by translating Haar probability measure on the symplectic group $\Sp_{2g}(\Z_\ell)$.
By \eqref{eq1a}, the cokernel of $I_{2g}-F$ models the $\ell$-primary groups $\mr{Pic}^0(C)_\ell$. The formal logarithm suggests a \emph{linearization} of this nonlinear model, replacing the random element $F$ by a random element of an affine translate of the symplectic Lie algebra
$$
\mf{sp}_{2g}(\Z_\ell) :=\{M\in\M_{2g}(\Z_\ell):M^{T}J_{2g}+J_{2g}M=0\}.
$$
Accordingly, the linear random matrix model of \cite{LST20} takes a Haar-random $M\in\mf{sp}_{2g}(\Z_\ell)$ and considers the cokernel of $M+\ell^e I_{2g}$.

Lipnowski, Sawin, and Tsimerman computed the limiting moments of the triple $(G,\omega_G,\psi_G)$ associated with the linear model $G=\cok(M+\ell^e I_{2g})$ \cite[Theorem 8.4]{LST20} and proved that its distribution converges to a limiting measure $\mu$ \cite[Theorem 8.7]{LST20}. Here $e=v_\ell(q-1)\ge1$ is fixed. After that, they showed that the corresponding triple $(G,\omega_G,\psi_G)$ associated with the nonlinear model $G=\cok(I_{2g}-F)$ has the same limiting moments \cite[Lemma 8.10]{LST20}, and the argument in \cite[Section 8.5]{LST20} gives convergence to the same limiting measure $\mu$ as $g\to\infty$. After forgetting the additional invariants $\omega_G$ and $\psi_G$, the two cokernels have the same limiting distribution.
Thus, linearization preserves the limiting cokernel distribution in this symplectic-similitude setting. The purpose of this paper is to establish an analogous result for the orthogonal group.

\subsection{Notation} \label{Sub12}

Let $\Alt_n(\Z_p)$ be the set of $n \times n$ alternating matrices over $\Z_p$, i.e.
$$
\Alt_n(\Z_p)=\{B\in\M_n(\Z_p):B^T=-B \}.
$$
Define the \emph{orthogonal group} and \emph{special orthogonal group} over $\Z_p$ by
$$
\mO_n(\Z_p)=\{A\in\GL_n(\Z_p):A^{T}A=I_n\},\quad \SO_n(\Z_p)=\ker(\det:\mO_n(\Z_p)\to\{\pm1\}).
$$
Also define the determinant $-1$ component by
$$
\mO_n^-(\Z_p)=\{A\in\mO_n(\Z_p):\det A=-1\},
$$
which satisfies $\mO_n(\Z_p)=\SO_n(\Z_p)\sqcup\mO_n^-(\Z_p)$. 

Set $L_n=\Z_p^n$ and write $e_1,\ldots,e_n$ for the standard basis of $L_n$. Regard $L_n$ as the standard symmetric $\Z_p$-lattice via $\langle \cdot, \, \cdot \rangle : L_n \times L_n \to \Z_p$, where $\langle x,y\rangle=x^{T}y$. More generally, a symmetric $\Z_p$-lattice is a finite free $\Z_p$-module $L$ equipped with a symmetric bilinear form. Such a lattice $L$ is \emph{unimodular} if the adjoint map
$$
L\rightarrow\Hom_{\Z_p}(L,\Z_p),\quad x\mapsto\langle x,-\rangle,
$$
is an isomorphism. For a unimodular symmetric $\Z_p$-lattice $L$, write $\mO(L)$ for its isometry group and $\SO(L)$ for the kernel of the determinant map $\det:\mO(L)\to\{\pm1\}$.

For a finitely generated $\Z_p$-module $M$, write
$$
\rank_{\Z_p}M=\dim_{\Q_p}(M\otimes_{\Z_p}\Q_p),\quad
M_{\tors}=\{x\in M:p^mx=0\text{ for some }m\ge1\}.
$$
A vector $v$ in a free $\Z_p$-module $L$ is \emph{primitive} if $v\notin pL$. Every finite abelian $p$-group is regarded as a finite $\Z_p$-module through its natural $\Z_p$-module structure. For abelian groups $G$ and $H$, write $\Hom(G,H)$ (resp. $\Sur(G,H)$) for the set of homomorphisms (resp. surjective homomorphisms) from $G$ to $H$, and write $\Aut(G)$ for the automorphism group of $G$. 
Let $H$ be a finite abelian $p$-group and $m$ be a positive integer. Set
$$
d(H)=\dim_{\F_p}(H/pH),\quad H[m]=\{h\in H:mh=0\}.
$$

If $H\cong\bigoplus_{i=1}^{r}\Z/p^{\lambda_i}\Z$ ($r=d(H)$) with $\lambda_1\ge\cdots\ge\lambda_r\ge1$, then $\lambda=(\lambda_1,\ldots,\lambda_r)$ is called the \emph{type} of $H$.
For a finite abelian $p$-group $H$, write $\Sym^2 H$ for the \emph{second symmetric power} of $H$:
$$
\Sym^2H=(H\otimes_{\Z}H)/\angles{x\otimes y-y\otimes x:x,y\in H}.
$$
We write $xy$ for the class of $x\otimes y$ in $\Sym^2 H$ and $x^2$ for $xx$. If $H$ has type $\lambda=(\lambda_1,\ldots,\lambda_r)$ and
$$
H\cong\bigoplus_{i=1}^r\Z/p^{\lambda_i}\Z\,h_i \quad (h_1, \ldots, h_r \in H),
$$
then
\begin{equation}\label{eq1b}
\Sym^2H\cong \bigoplus_{1\le i\le j\le r} \Z/p^{\min(\lambda_i,\lambda_j)}\Z\,h_ih_j
\end{equation}
so we have
$$
|\Sym^2H|
=p^{\sum_{1\le i\le j\le r}\min(\lambda_i,\lambda_j)}
=p^{\lambda_1+2\lambda_2+\cdots+r\lambda_r}.
$$

A finite abelian $p$-group $M$ is called \emph{square} if $M\cong G\oplus G$ for some finite abelian $p$-group $G$. For a bilinear pairing $\lambda:M\times M\to\Q_p/\Z_p$, \emph{skew-symmetric} means $\lambda(x,y)=-\lambda(y,x)$, \emph{alternating} means $\lambda(x,x)=0$ for every $x$, and \emph{perfect} means that
$$
M\rightarrow\Hom(M,\Q_p/\Z_p),\quad x\mapsto\lambda(x,-),
$$
is an isomorphism. The group $M$ is called \emph{symplectic} if it admits a perfect alternating pairing. Such a pairing exists if and only if $M$ is square \cite[Proposition 2]{Del01}. For a symplectic $M$, write $\Sp(M)$ for the automorphism group preserving a fixed perfect alternating pairing. In particular, when $p$ is odd, the symplectic finite abelian $p$-groups are precisely the square ones, and every skew-symmetric pairing is alternating.

Next we provide some notation used only in the case $p=2$. For a finite abelian $2$-group $H$, write $t(H)$ for the number of direct summands isomorphic to $\Z/2\Z$ in its elementary-divisor decomposition; equivalently,
$$
t(H)=\dim_{\F_2}(H[2]/(H[2]\cap 2H)).
$$
We write $\Gamma(H)$ for Whitehead's universal quadratic group; see Section \ref{Sub61} for its properties. 
For $\delta\in\{0,1\}$, set
\begin{align*}
\mc S_{\sq}^{(\delta)} &=\{\Z_2^\delta\oplus G\oplus G:
G\text{ is a finite abelian }2\text{-group}\}, \\
\mc S_{\nsq}^{(\delta)} &=\{\Z_2^\delta\oplus\Z/2\Z\oplus G\oplus G: G\text{ is a finite abelian }2\text{-group}\}.
\end{align*}
In both definitions, the elements are understood up to isomorphism.

Finally, we record the probabilistic notation used throughout the paper. For a random variable $X$ and an event $E$, write $\bE(X)$ for the expected value of $X$, $\bP(E)$ for the probability of $E$, and $\mathbf{1}_E$ for its indicator. If $Y$ is a discrete random variable and $\bP(Y=y)>0$, write $\bE(X\mid Y=y)$ and $\bP(E\mid Y=y)$ for the corresponding conditional expected value and conditional probability. Unless otherwise specified, a random element of a compact group is distributed according to its Haar probability measure. By Haar measure on $\mO_n^-(\Z_p)$, we mean the normalized restriction of the Haar measure on $\mO_n(\Z_p)$. For a random finitely generated $\Z_p$-module $X$ and a finite abelian $p$-group $H$, the \emph{$H$-moment} of $X$ is the expected value
$$
\bE(\#\Sur(X,H)).
$$

For a sequence of random finitely generated $\Z_p$-modules $(X_n)_{n\ge1}$ and a probability law $\mu$ on the set of isomorphism classes of finitely generated $\Z_p$-modules, we write $X_n\dist\mu$ if
$$
\lim_{n \to \infty}\bP(X_n\cong H)=\mu(H)
$$
for every finitely generated $\Z_p$-module $H$. On the countable discrete space of isomorphism classes of finitely generated $\Z_p$-modules, this pointwise convergence is equivalent to weak convergence of the distributions of $X_n$ to $\mu$. We call a moment formula or orbit classification \emph{stable} if, for each fixed finite abelian $p$-group $H$, it holds for all sufficiently large $n$ depending on $H$.

\subsection{Main results} \label{Sub13}

\begin{thm}\label{thm1a}
Let $H$ be a finite abelian $p$-group.
\begin{enumerate}
\item Let $p$ be odd, and let $A_n$ be a Haar-random matrix in either $\mO_n(\Z_p)$, $\SO_n(\Z_p)$ or $\mO_n^{-}(\Z_p)$. Then for all sufficiently large $n$ (depending on $H$),
$$
\bE(\#\Sur(\cok(A_n-I_n),H))=|\Sym^2H|.
$$

\item Let $p=2$, and let $A_n$ be a Haar-random matrix in either $\mO_n(\Z_p)$, $\SO_n(\Z_p)$ or $\mO_n^{-}(\Z_p)$. Then for all sufficiently large $n$ (depending on $H$),
$$
\bE(\#\Sur(\cok(A_n-I_n),H))=|\Gamma(H)|\left(\frac{3}{2}-2^{-t(H)-2}-2^{-d(H)-2}\right).
$$
\end{enumerate}
\end{thm}

Nguyen and Wood \cite[Theorem 3.1]{NW25} proved that if $B_n$ is Haar-random in $\Alt_n(\Z_p)$, then the limiting $H$-moment of $\cok(B_n)$ is also $|\Sym^2 H|$. Thus, for odd $p$, Theorem \ref{thm1a}(a) shows that the orthogonal and alternating models have the same limiting $H$-moments.

For a prime $p$ and $\delta \in \{0,1\}$, set
$$
c_{p,\delta}=\prod_{i=\delta}^{\infty}(1-p^{-(2i+1)}).
$$
When $p$ is odd, define probability laws $\mu_{p,0}$ and $\mu_{p,1}$ on the set of isomorphism classes of finitely generated $\Z_p$-modules by
$$
\mu_{p,0}(M)=\begin{cases}
c_{p,0}|M|/|\Sp(M)|,&M\text{ is symplectic},\\
0,&\text{otherwise},
\end{cases}
$$
and
$$
\mu_{p,1}(M)=\begin{cases}
c_{p,1}/|\Sp(G)|,& M\cong \Z_p\oplus G\text{ for some symplectic }p\text{-group }G,\\
0,&\text{otherwise}.
\end{cases}
$$
For a Haar-random matrix $B_n\in\Alt_n(\Z_p)$, Bhargava, Kane, Lenstra, Poonen and Rains \cite[Theorems 3.9 and 3.11]{BKLPR15} proved that as $g \to \infty$,
$$
\cok(B_{2g})\dist\mu_{p,0},\quad
\cok(B_{2g+1})\dist\mu_{p,1}.
$$
We note that for the odd-dimensional case, $\cok(B_{2g+1})$ has $\Z_p$-rank $1$ almost surely by \cite[Proposition 2.1(b)]{BKLPR15}.
The next theorem shows that the orthogonal model has the same parity-dependent limiting laws, with the two parities interchanged on the negative determinant component.

\begin{thm}\label{thm1b}
Let $p$ be an odd prime.
\begin{enumerate}
\item Let $A_n$ be a Haar-random matrix in $\SO_n(\Z_p)$. Then as $g \to \infty$,
$$
\cok(A_{2g}-I_{2g})\dist\mu_{p,0},\quad
\cok(A_{2g+1}-I_{2g+1})\dist\mu_{p,1}.
$$

\item Let $A_n$ be a Haar-random matrix in $\mO_n^-(\Z_p)$. Then as $g \to \infty$,
$$
\cok(A_{2g}-I_{2g})\dist\mu_{p,1},\quad
\cok(A_{2g+1}-I_{2g+1})\dist\mu_{p,0}. 
$$

\item Let $A_n$ be a Haar-random matrix in $\mO_n(\Z_p)$. Then as $n \to \infty$,
$$
\cok(A_n-I_n)\dist\frac{1}{2}\mu_{p,0}+\frac{1}{2}\mu_{p,1}.
$$
\end{enumerate}
\end{thm}

For $p=2$, Theorem \ref{thm1a}(b) gives the unconditional stable moments. However, the structure of $\cok(A_n-I_n)$ depends on the parity
$$
\ve(A_n):=v_2(|\cok(A_n-I_n)_{\tors}|) \pmod{2},
$$
and the corresponding limiting distributions also depend on this parity. We refer to Section \ref{Sub91} for the convention of the \emph{spinor norm} $\mathrm{sp}_{-}$ in this paper.

\begin{thm} \label{thm1c}
Let $n \ge 2$, let $A_n$ be a Haar-random matrix in $\SO_n(\Z_2)$ and set $\ve = \ve(A_n)$. 
Then $\ve(A_n)=v_2(\mathrm{sp}_{-}(A_n)) \pmod{2}$, and each of its two values has probability $1/2$. Moreover, for every finite abelian $2$-group $H$ and all sufficiently large $n$,
\begin{equation}\label{eq1c}
\begin{aligned}
\bE(\#\Sur(\cok(A_n-I_n),H)\mid\ve=0)
&=|\Gamma(H)|\left(\frac{3}{2}-2^{-d(H)-1}\right),\\
\bE(\#\Sur(\cok(A_n-I_n),H)\mid\ve=1)
&=|\Gamma(H)|\left(\frac{3}{2}-2^{-t(H)-1}\right).
\end{aligned}
\end{equation}
\end{thm}
We now identify the limiting distributions corresponding to the two conditional moment formulas in Theorem \ref{thm1c}. Let $M$ be a finite symplectic $2$-group and $\delta\in\{0,1\}$. Write
$$
M_\delta=\Z_2^\delta\oplus M,\quad
w_\delta(M)=c_{2,\delta}\frac{|M|^{1-\delta}}{|\Sp(M)|}.
$$
Define nonnegative functions $\mu_{\sq}^{(\delta)}$ and $\mu_{\nsq}^{(\delta)}$ on the isomorphism classes of finitely generated $\Z_2$-modules by
\begin{align*}
\mu_{\sq}^{(\delta)}(M_\delta)
&=\frac{\#\Sur(M_\delta,\Z/2\Z)}{2} w_\delta(M),\\
\mu_{\nsq}^{(\delta)}(\Z/2\Z\oplus M_\delta)
&=w_\delta(M),
\end{align*}
with $\mu_{\sq}^{(\delta)}$ and $\mu_{\nsq}^{(\delta)}$ supported on $\mc S_{\sq}^{(\delta)}$ and $\mc S_{\nsq}^{(\delta)}$, respectively. Finally, let
\begin{equation*}
\mu_{2,\delta}
=\frac{1}{2}\mu_{\sq}^{(\delta)}
+\frac{1}{2}\mu_{\nsq}^{(\delta)}.
\end{equation*}

\begin{thm} \label{thm1d}
Let $A_n$ and $\ve=\ve(A_n)$ be as in Theorem \ref{thm1c}, and let $\delta \in \{0,1\}$.
\begin{enumerate}
\item The functions $\mu_{\sq}^{(\delta)}$ and $\mu_{\nsq}^{(\delta)}$ are probability laws on the isomorphism classes of finitely generated $\Z_2$-modules, supported on $\mc S_{\sq}^{(\delta)}$ and $\mc S_{\nsq}^{(\delta)}$, respectively.

\item For every finite symplectic $2$-group $M$,
\begin{align*}
\lim_{\substack{n\to\infty\\ n\equiv\delta \pmod{2}}}
\bP(\cok(A_n-I_n)\cong M_\delta \mid\ve=0)
&=\frac{\#\Sur(M_\delta,\Z/2\Z)}{2}w_{\delta}(M), \\
\lim_{\substack{n\to\infty\\ n\equiv\delta \pmod{2}}}
\bP(\cok(A_n-I_n)\cong \Z/2\Z\oplus M_\delta \mid\ve=1)
&=w_{\delta}(M).
\end{align*}
The conditional limiting probabilities are zero outside $\mc S_{\sq}^{(\delta)}$ and $\mc S_{\nsq}^{(\delta)}$, respectively.

\item For every finite symplectic $2$-group $M$,
\begin{align*}
\lim_{\substack{n\to\infty\\ n\equiv\delta \pmod{2}}}
\bP(\cok(A_n-I_n)\cong M_\delta)
&=\mu_{2,\delta}(M_\delta)=\frac{\#\Sur(M_\delta,\Z/2\Z)}{4} w_\delta(M),\\
\lim_{\substack{n\to\infty\\ n\equiv\delta \pmod{2}}}
\bP(\cok(A_n-I_n)\cong\Z/2\Z\oplus M_\delta)
&=\mu_{2,\delta}(\Z/2\Z\oplus M_\delta)=\frac{1}{2}w_\delta(M).
\end{align*}
The limiting probability is zero outside $\mc S_{\sq}^{(\delta)} \sqcup \mc S_{\nsq}^{(\delta)}$.

\item The probability laws $\mu_{\sq}^{(\delta)}$ and $\mu_{\nsq}^{(\delta)}$ are the unique probability laws on the isomorphism classes of finitely generated $\Z_2$-modules supported on $\mc S_{\sq}^{(\delta)}$ and $\mc S_{\nsq}^{(\delta)}$, respectively, whose $H$-moments are the two right-hand sides of \eqref{eq1c} for every finite abelian $2$-group $H$.
\end{enumerate}
\end{thm}

\subsection{Organization}\label{Sub14}

The paper is organized as follows. Section \ref{Sec2} gives a preliminary linearization result for the special linear group and its Lie algebra. Section \ref{Sec3} determines the almost-sure $\Z_p$-rank of $\cok(A_n-I_n)$ for a Haar-random matrix $A_n$ in $\SO_n(\Z_p)$ and $\mO_n^-(\Z_p)$.
In Section \ref{Sec4}, we classify the stable $\mO_n(\Z_p)$-orbits on the set of surjections from $L_n$ to a finite abelian $p$-group $H$ when $p$ is odd, and apply Burnside's lemma to prove Theorem \ref{thm1a}(a). In Section \ref{Sec5}, we construct a perfect skew-symmetric pairing on $\cok(A_n-I_n)_{\tors}$, show that it is symplectic when $p$ is odd, and prove Theorem \ref{thm1b}.

Sections \ref{Sec6}--\ref{Sec10} focus on the case $p=2$ and prove Theorems \ref{thm1a}(b), \ref{thm1c} and \ref{thm1d}. Sections \ref{Sec6} and \ref{Sec7} introduce the dyadic quadratic orbit invariant, classify the stable orthogonal orbits, and count them. Section \ref{Sec8} determines the possible dyadic cokernel types, and Section \ref{Sec9} identifies the parity of the $2$-adic valuation of the order of the torsion subgroup with spinor parity and computes the two parity-resolved moment families, proving Theorems \ref{thm1a}(b) and \ref{thm1c}. Finally, we show that the probability laws $\mu_{\sq}^{(\delta)}$ and $\mu_{\nsq}^{(\delta)}$ in Theorem \ref{thm1d} are uniquely determined by their $H$-moments for each $\delta\in\{0,1\}$ and prove Theorem \ref{thm1d} in Section \ref{Sec10}.

\section{The special linear model and its linearization} \label{Sec2}

Before turning to the orthogonal group, we first consider the special linear group as a toy model and show that its linearization preserves the limiting cokernel distribution. Let
$$
\SL_n(\Z_p)=\{A\in\GL_n(\Z_p):\det A=1\}
$$
be the special linear group, with Lie algebra
$$
\mf{sl}_n(\Z_p)=\{X\in\M_n(\Z_p):\tr(X)=0\}.
$$

\begin{lem}\label{lem2a}
Let $H$ be a finite abelian $p$-group. If $n>d(H)$, then $\SL_n(\Z_p)$ acts transitively on $\Sur(L_n,H)$ by precomposition.
\end{lem}

\begin{proof}
Let $d=d(H)$, $\lambda=(\lambda_1,\ldots,\lambda_d)$ be the type of $H$, and write $H=\bigoplus_{i=1}^d\Z/p^{\lambda_i}\Z\,h_i$.
Define $\phi_0\in\Sur(L_n,H)$ by $\phi_0(e_i)=h_i$ for $1\le i\le d$ and $\phi_0(e_i)=0$ for $i>d$. Fix $\phi\in\Sur(L_n,H)$ and choose $x_i\in L_n$ such that $\phi(x_i)=h_i$ for $1\le i\le d$. Then the reductions of $x_1,\ldots,x_d$ modulo $p$ are linearly independent, since their images form a basis of $H/pH$. Hence they can be extended to a $\Z_p$-basis $x_1, \ldots, x_d, y_{d+1}, \ldots, y_n$ of $L_n$. 

For each $d+1 \le j \le n$, let $\phi(y_j)=\sum_{i=1}^{d} a_{ij}h_i$ ($a_{ij} \in \Z$) and set $y_j' = y_j - \sum_{i=1}^{d} a_{ij}x_i \in \ker \phi$. For a matrix $C \in \GL_n(\Z_p)$ given by
$$
Ce_i = x_i \;\; (1 \le i \le d) \quad \text{and} \quad Ce_i = y_i' \;\; (d+1 \le i \le n),
$$
we have $(\phi C)(e_i)=\phi_0(e_i)$ for every $1 \le i \le n$ so $\phi C = \phi_0$. 
Since $n>d$, a matrix 
$$
U=\mathrm{diag}(\underbrace{1,\ldots,1}_{n-1},\det(C)^{-1}) \in \GL_n(\Z_p)
$$ 
satisfies $CU\in\SL_n(\Z_p)$ and $\phi CU=\phi_0 U=\phi_0$.
\end{proof}

\begin{prop}\label{prop2b}
Let $H$ be a finite abelian $p$-group, and let $A_n$ and $B_n$ be Haar-random matrices in $\SL_n(\Z_p)$ and $\mf{sl}_n(\Z_p)$, respectively. If $n>d(H)$, then
$$
\bE(\#\Sur(\cok(A_n-I_n),H))=1
$$
and
$$
\bE(\#\Sur(\cok(B_n),H))=\prod_{i=0}^{d(H)-1}(1-p^{i-n}).
$$
\end{prop}

\begin{proof}
For the first identity, we adapt the orbit-stabilizer calculation in the proof of \cite[Theorem 3.1]{LT19}. A surjection from $\cok(A_n-I_n)$ to $H$ is equivalent to a map $\phi\in\Sur(L_n,H)$ satisfying $\phi A_n=\phi$. By Lemma \ref{lem2a}, the elements of $\Sur(L_n,H)$ form a single $\SL_n(\Z_p)$-orbit. Hence the orbit-stabilizer formula gives
$$
[\SL_n(\Z_p):\Stab(\phi)]=\#\Sur(L_n,H),
$$
where $\Stab(\phi)=\{A\in\SL_n(\Z_p):\phi A=\phi\}$. By the translation-invariance of Haar measure,
$$
\bP(\phi A_n=\phi)=\frac{1}{\#\Sur(L_n,H)}
$$
for every $\phi\in\Sur(L_n,H)$, and consequently,
$$
\bE(\#\Sur(\cok(A_n-I_n),H))=\sum_{\phi\in\Sur(L_n,H)}\bP(\phi A_n=\phi)=1.
$$

For the second identity, fix $\phi\in\Sur(L_n,H)$ and consider the homomorphism
$$
T_\phi:\mf{sl}_n(\Z_p)\rightarrow\Hom(L_n,H)
\quad (X\mapsto\phi X).
$$
Since $n>d:=d(H)$, $\ker\phi$ contains a primitive vector $v\in\ker\phi$. Choose $\ell\in\Hom(L_n,\Z_p)$ such that $\ell(v)=1$. Given $f\in\Hom(L_n,H)$, choose $X_0\in\End(L_n)$ such that $\phi X_0=f$ and set
$$
X=X_0-\tr(X_0)v\ell,
$$
where $v\ell\in\End(L_n)$ is given by $x \mapsto \ell(x)v$. Since $\tr(v\ell)=\ell(v)=1$ and $\phi(v\ell)=0$, we have $X\in\mf{sl}_n(\Z_p)$ and $T_\phi(X)=\phi X_0=f$. Thus $T_\phi$ is surjective. 

It follows that $T_\phi(B_n)$ is uniform on $\Hom(L_n,H)$, and hence
$$
\bP(\phi B_n=0)=\frac{1}{\#\Hom(L_n,H)}.
$$
Consequently,
$$
\bE(\#\Sur(\cok(B_n),H))=\sum_{\phi\in\Sur(L_n,H)}\bP(\phi B_n=0)=\frac{\#\Sur(L_n,H)}{\#\Hom(L_n,H)}.
$$
By Nakayama's lemma, a homomorphism from $L_n$ to $H$ is surjective if and only if its modulo $p$ reduction from $\F_p^n$ to $\F_p^d$ is surjective. Hence
\begin{equation*}
\frac{\#\Sur(L_n,H)}{\#\Hom(L_n,H)}
=\frac{\#\Sur(\F_p^n, \F_p^d)}{\#\Hom(\F_p^n, \F_p^d)}
=\frac{1}{p^{dn}} \prod_{i=0}^{d-1} (p^n-p^i)
=\prod_{i=0}^{d-1}(1-p^{i-n}). \qedhere
\end{equation*}
\end{proof}

\begin{thm}\label{thm2c}
Let $A_n$ and $B_n$ be as in Proposition \ref{prop2b}. Then for every finite abelian $p$-group $G$,
$$
\lim_{n\to\infty}\bP(\cok(A_n-I_n)\cong G)
=\lim_{n\to\infty}\bP(\cok(B_n)\cong G)
=\frac{1}{|\Aut(G)|}\prod_{i=1}^\infty(1-p^{-i}).
$$
\end{thm}

\begin{proof}
By Proposition \ref{prop2b}, for every finite abelian $p$-group $H$,
$$
\lim_{n\to\infty}\bE(\#\Sur(\cok(A_n-I_n),H))
=\lim_{n\to\infty}\bE(\#\Sur(\cok(B_n),H))=1.
$$ 
Wood's moment theorem \cite[Theorem 3.1]{Woo19}, together with the argument of \cite[Corollary 3.4]{Woo19}, completes the proof.
\end{proof}

\section{Almost-sure \texorpdfstring{$\Z_p$}{Zp}-rank of the orthogonal model}\label{Sec3}

In this section, we determine the almost-sure $\Z_p$-rank of $\cok(A_n-I_n)$ for a Haar-random matrix $A_n$ in either $\SO_n(\Z_p)$ or $\mO_n^-(\Z_p)$, according to whether $\det A_n=(-1)^n$ or $\det A_n=(-1)^{n+1}$. In particular, if $A_n$ (resp. $B_n$) is Haar-random in $\SO_n(\Z_p)$ (resp. $\Alt_n(\Z_p)$), then $\cok(A_n-I_n)$ and $\cok(B_n)$ have the same almost-sure $\Z_p$-rank.

\begin{prop}\label{prop3a}
Let $A_n$ be a Haar-random matrix in either $\SO_n(\Z_p)$ or $\mO_n^-(\Z_p)$. Then almost surely,
$$
\rank_{\Z_p}\cok(A_n-I_n)=
\begin{cases}
0,& \det A_n=(-1)^n,\\
1,& \det A_n=(-1)^{n+1}.
\end{cases}
$$
In the latter case, $1$ is a simple root of the characteristic polynomial of $A_n$ almost surely.
\end{prop}

\begin{proof}
For $A_n\in\mO_n(\Z_p)$, consider its characteristic polynomial
$$
P_{A_n}(t)=\det(tI_n-A_n) \in \Z_p[t]
$$
and write $V=\Q_p^n$. Then
$$
\rank_{\Z_p}\cok(A_n-I_n)=\dim_{\Q_p}\ker(A_n-I_n:V\to V).
$$
Since $A_n^{-1}=A_n^T$, we have $P_{A_n^{-1}}(t)=P_{A_n}(t)$. Hence
\begin{equation} \label{eq3a}
\begin{split}
P_{A_n}(t) &=\det(-tA_n(t^{-1}I_n-A_n^{-1})) \\
&=(-1)^nt^n\det(A_n)P_{A_n^{-1}}(t^{-1}) \\
&=(-1)^nt^n\det(A_n)P_{A_n}(t^{-1}).
\end{split}
\end{equation}
Let $\mc{M}_n$ denote either $\SO_n$ or $\mO_n^-$. By \cite[Proposition 2.1(b)]{BKLPR15}, the $\Z_p$-points of any proper Zariski-closed subset of $\mc{M}_n$ have Haar measure zero in $\mc{M}_n(\Z_p)$.

Suppose first that $\det A_n=(-1)^n$: that is, $n$ is even if $\mc{M}_n=\SO_n$, and $n$ is odd if $\mc{M}_n=\mO_n^-$. Then the locus
$$
\{A \in \mc{M}_n(\Z_p):\rank_{\Z_p}\cok(A-I_n)\ge1\}
=\{A \in \mc{M}_n(\Z_p):\det(A-I_n)=0\}
$$
is the set of $\Z_p$-points of a Zariski-closed subset of $\mc{M}_n$. This subset is proper, since $-I_n \in \mc{M}_n(\Z_p)$ and $\det(-I_n-I_n)=(-2)^n \ne 0$. Hence $\rank_{\Z_p}\cok(A_n-I_n)=0$ almost surely.

Now suppose that $\det A_n=(-1)^{n+1}$: that is, $n$ is odd if $\mc{M}_n=\SO_n$, and $n$ is even if $\mc{M}_n=\mO_n^-$. Then \eqref{eq3a} gives $P_{A_n}(1)=0$, so $\rank_{\Z_p}\cok(A_n-I_n)\ge1$. The locus
$$
\{A \in \mc{M}_n(\Z_p):\rank_{\Z_p}\cok(A-I_n)\ge2\}
=\{A \in \mc{M}_n(\Z_p):\rank_{\Q_p}(A-I_n)\le n-2\}
$$
is the set of $\Z_p$-points of a Zariski-closed subset of $\mc{M}_n$. This subset is proper, since
$$
D_n=\mathrm{diag}(1, \underbrace{-1,\ldots,-1}_{n-1}) \in\mc{M}_n(\Z_p)
$$
and $\rank_{\Q_p}(D_n-I_n)=n-1$. Hence $\rank_{\Z_p}\cok(A_n-I_n)=1$ almost surely.

It remains to show that $t=1$ is a simple root of $P_{A_n}(t)$ almost surely, under the assumption $\det A_n=(-1)^{n+1}$. The locus
$$
\{A \in \mc{M}_n(\Z_p):P_{A}'(1)=0\}
$$
is the set of $\Z_p$-points of a Zariski-closed subset of $\mc{M}_n$. This subset is proper, since $D_n \in \mc{M}_n(\Z_p)$ defined above satisfies
$$
P_{D_n}(t)=(t-1)(t+1)^{n-1},\quad 
P_{D_n}'(1)=2^{n-1} \ne 0.
$$
Hence $P_{A_n}'(1)\ne0$ almost surely, so $t=1$ is a simple root of $P_{A_n}(t)$ almost surely.
\end{proof}

\section{Stable orthogonal orbits at odd primes}\label{Sec4}

Throughout this section, assume that $p$ is odd. We classify the $\mO_n(\Z_p)$-orbits on $\Sur(L_n,H)$ for a finite abelian $p$-group $H$ and use this classification to prove Theorem \ref{thm1a}(a).

\subsection{The symmetric-square orbit invariant}\label{Sub41}

We first record a form of Burnside's lemma for a transitive action of a compact group on a finite set.

\begin{lem}\label{lem4a}
Let $G$ be a compact group acting continuously and transitively on a finite set $\Omega$, let $K$ be the stabilizer of an element of $\Omega$, and let $\psi:G\to\{\pm1\}$ be a continuous character. Then
$$
\int_G\psi(g)\#\Fix_\Omega(g)\,dg
=
\begin{cases}
1,& \psi|_K=1,\\
0,& \psi|_K\ne1,
\end{cases}
$$
where $\Fix_\Omega(g)$ denotes the set of fixed points of $g$ in $\Omega$.
\end{lem}

\begin{proof}
Let $\omega\in\Omega$ and $G_\omega$ be its stabilizer. Since
$$
\#\Fix_\Omega(g)=\sum_{\omega\in\Omega}\mathbf{1}_{G_\omega}(g)
$$
for every $g \in G$, we have
$$
\int_G\psi(g)\#\Fix_\Omega(g)\,dg =\sum_{\omega\in\Omega}\int_{G_\omega}\psi(g)\,dg.
$$
Since the action of $G$ on $\Omega$ is transitive, each $G_\omega$ is conjugate to $K$ and has Haar measure $|\Omega|^{-1}$. If $\psi|_K=1$, each integral on the right hand side is $|\Omega|^{-1}$, so their sum is $1$. If $\psi|_K\ne1$, choose $k\in K$ with $\psi(k)=-1$. The translation-invariance of Haar measure gives
$$
\int_K\psi(g)\,dg =\int_K\psi(kg)\,dg =-\int_K\psi(g)\,dg,
$$
so each integral on the right hand side is zero.
\end{proof}

Define a function $\Theta_n : \Sur(L_n,H) \to \Sym^2H$ by
$$
\Theta_n(f)=\sum_{k=1}^nf(e_k)^2\in\Sym^2H.
$$
Recall that we write $xy$ for the class of $x \otimes y$ in $\Sym^2H$. We first show that $\Theta_n$ is invariant under the action of $\mO_n(\Z_p)$ by precomposition. Let $U=(a_{ij})\in\mO_n(\Z_p)$ and set $x_i=f(e_i)$. Then
$$
\begin{aligned}
\Theta_n(f\circ U) &=\sum_{j=1}^n\left(\sum_{i=1}^na_{ij}x_i\right)^2\\
&=\sum_{i=1}^n\left(\sum_{j=1}^na_{ij}^2\right)x_i^2+2\sum_{1\le i<i'\le n} \left(\sum_{j=1}^na_{ij}a_{i'j}\right)x_ix_{i'}\\
&=\sum_{i=1}^nx_i^2=\Theta_n(f),
\end{aligned}
$$
where the third equality follows from $UU^T=I_n$. Thus $\Theta_n$ is constant on each $\mO_n(\Z_p)$-orbit.

In the remainder of this section, let $\lambda=(\lambda_1,\ldots,\lambda_r)$ be the type of $H$ and write
$$
H=\bigoplus_{i=1}^r\Z/p^{\lambda_i}\Z\,h_i.
$$
The standard bilinear form identifies $L_n=\Z_p^n$ with $\Hom_{\Z_p}(L_n,\Z_p)$. For each $1\le i\le r$, reduction modulo $p^{\lambda_i}$ gives an isomorphism
$$
L_n/p^{\lambda_i}L_n \cong \Hom_{\Z_p}(L_n,\Z/p^{\lambda_i}\Z),
\quad u+p^{\lambda_i}L_n \mapsto(x\mapsto\langle u,x\rangle\bmod p^{\lambda_i}).
$$
Consequently, every $f\in\Sur(L_n,H)$ can be written as
\begin{equation}\label{eq4a}
f(x)=\sum_{i=1}^r
(\langle u_i,x\rangle\bmod p^{\lambda_i})h_i
\end{equation}
for some $u_1,\ldots,u_r \in L_n$, where $u_i+p^{\lambda_i}L_n \in L_n/p^{\lambda_i}L_n$ is uniquely determined by $f$. A choice of representatives $(u_1,\ldots,u_r)$ is called a \emph{frame representing $f$}. 

If $U\in\mO_n(\Z_p)$ and $F\in\M_{r\times n}(\Z_p)$ is the matrix whose $i$-th row is $u_i^T$, then the $i$-th row of $FU$ is $u_i^TU=(U^Tu_i)^T$. Moreover,
$$
(f\circ U)(x)=\sum_{i=1}^r
(\langle U^Tu_i,x\rangle\bmod p^{\lambda_i})h_i,
$$
so $(U^Tu_1,\ldots,U^Tu_r)$ is a frame representing $f\circ U$. By Nakayama's lemma, $f$ is surjective if and only if the reductions $\bar{u}_1,\ldots,\bar{u}_r$ modulo $p$ are linearly independent in $L_n/pL_n$.

\begin{lem}\label{lem4b}
For every finite abelian $p$-group $H$, the map
$$
\Theta_n:\Sur(L_n,H)\to\Sym^2H
$$
is surjective for all sufficiently large $n$.
\end{lem}

\begin{proof}
By \eqref{eq1b}, the elements $xy$ with $x,y\in H$ generate $\Sym^2H$. Since $p$ is odd,
$$
xy=\left(\frac{x+y}{2}\right)^2-\left(\frac{x-y}{2}\right)^2
=\left(\frac{x+y}{2}\right)^2+(|\Sym^2H|-1)\left(\frac{x-y}{2}\right)^2
$$
in a finite group $\Sym^2H$. Hence every element of $\Sym^2H$ is a sum of squares. Since $\Sym^2H$ is finite, there is an integer $m\ge1$ such that every element of $\Sym^2H$ is a sum of at most $m$ squares.

Fix $\theta\in\Sym^2H$ and choose $x_1,\ldots,x_m\in H$ such that $\theta-\sum_{i=1}^rh_i^2=\sum_{j=1}^mx_j^2$. For every $n\ge r+m$, define $f \in \Hom(L_n,H)$ by
$$
f\left( \sum_{i=1}^{n} a_ie_i \right)=\sum_{i=1}^{r} a_ih_i+\sum_{i=r+1}^{r+m} a_ix_{i-r}.
$$
Then $f$ is surjective and $\Theta_n(f)=\sum_{i=1}^rh_i^2+\sum_{i=r+1}^{r+m}x_{i-r}^2=\theta$.
\end{proof}

\begin{lem}\label{lem4c}
Let $H$ be a finite abelian $p$-group and $f,g \in \Sur(L_n,H)$ be represented by frames $(u_1,\ldots,u_r)$ and $(v_1,\ldots,v_r)$, respectively. Then $\Theta_n(f)=\Theta_n(g)$ if and only if
\begin{equation}\label{eq4b}
\langle u_i,u_j\rangle\equiv\langle v_i,v_j\rangle
\pmod{p^{\min(\lambda_i,\lambda_j)}}
\end{equation}
for every $1\le i,j\le r$.
\end{lem}

\begin{proof}
By \eqref{eq4a} and the definition of $\Theta_n$, we have
\begin{align*}
\Theta_n(f)
&=\sum_{i=1}^r
\left(\sum_{k=1}^n\langle u_i,e_k\rangle^2\right)h_i^2
+2\sum_{1\le i<j\le r}
\left(\sum_{k=1}^n
\langle u_i,e_k\rangle\langle u_j,e_k\rangle\right)h_ih_j\\
&=\sum_{i=1}^r\langle u_i,u_i\rangle h_i^2
+2\sum_{1\le i<j\le r}\langle u_i,u_j\rangle h_ih_j.
\end{align*}
The lemma follows from \eqref{eq1b}, since $2$ is a unit in $\Z/p^{\min(\lambda_i,\lambda_j)}\Z$ for every $1\le i<j\le r$.
\end{proof}

\subsection{Stable orbit classification and moment computation}\label{Sub42}

We next show that the congruences in \eqref{eq4b} can be lifted to equalities in $\Z_p$. For $u,v\in L_n$ and $k \ge 1$, we write $u\equiv v\pmod{p^k}$ if $u-v\in p^k L_n$.

\begin{lem}\label{lem4d}
Let $f,g\in\Sur(L_n,H)$, and let $(u_1,\ldots,u_r)$ and $(v_1,\ldots,v_r)$ be frames representing $f$ and $g$, respectively. Assume that \eqref{eq4b} holds. Then there are vectors $\wtil{u}_1,\ldots,\wtil{u}_r\in L_n$ such that
\begin{equation} \label{eq4c}
\wtil{u}_i\equiv u_i \pmod{p^{\lambda_i}}, \quad \langle \wtil{u}_i,\wtil{u}_j\rangle
=\langle v_i,v_j\rangle \quad (1\le i,j \le r).
\end{equation}
\end{lem}

\begin{proof}
Set $u_i^{(0)}=u_i$. We inductively construct the vectors $u_i^{(m)}\in L_n$ ($m=0, 1, \ldots$) such that
$$
u_i^{(m)} \equiv u_i\pmod{p^{\lambda_i}}
$$
for all $1 \le i \le r$ and
$$
D_{ij}^{(m)}:=\langle v_i,v_j\rangle-\langle u_i^{(m)},u_j^{(m)}\rangle \in p^{\min(\lambda_i,\lambda_j)+m}\Z_p
$$
for all $1\le i,j\le r$. The case $m=0$ follows from \eqref{eq4b}. 

Assume that $u_i^{(m)}$ have been constructed for each $1 \le i \le r$. For vectors $z_1,\ldots,z_r\in L_n$ that will be specified below, let
$$
u_i^{(m+1)}=u_i^{(m)}+p^{\lambda_i+m}z_i.
$$
Define the elements $a_i \in \F_p$ ($1 \le i \le r$), $b_{ij} \in \F_p$ ($1 \le i<j \le r$) by
$$
a_i:=\frac{D_{ii}^{(m)}}{p^{\lambda_i+m}} \;\; \bmod{p},\quad b_{ij}:=\frac{D_{ij}^{(m)}}{p^{\lambda_j+m}} \;\; \bmod{p}.
$$
Note that $\lambda_1 \ge \cdots \ge \lambda_r$ so $\min(\lambda_i,\lambda_j)=\lambda_j$ if $i<j$. By the definition of $u_i^{(m+1)}$,
$$
D_{ij}^{(m)}-D_{ij}^{(m+1)}
=p^{\lambda_i+m}\langle z_i,u_j^{(m)}\rangle +p^{\lambda_j+m}\langle u_i^{(m)},z_j\rangle +p^{\lambda_i+\lambda_j+2m}\langle z_i,z_j\rangle.
$$
Since $u_i^{(m)}\equiv u_i\pmod{p^{\lambda_i}}$, $u_i^{(m)}$ and $u_i$ have the same reduction modulo $p$. Thus
\begin{align*}
\frac{D_{ij}^{(m)}-D_{ij}^{(m+1)}}{p^{\lambda_j+m}} \;\; \bmod p
&=\begin{cases}
\langle\bar{u}_i,\bar{z}_j\rangle, & \lambda_i>\lambda_j,\\
\langle\bar{u}_i,\bar{z}_j\rangle+\langle\bar{u}_j,\bar{z}_i\rangle, & \lambda_i=\lambda_j
\end{cases}
\;\; (1 \le i<j \le r),\\
\frac{D_{ii}^{(m)}-D_{ii}^{(m+1)}}{p^{\lambda_i+m}} \;\; \bmod p
&=2\langle\bar{u}_i,\bar{z}_i\rangle \;\; (1 \le i \le r).
\end{align*}

Set $x_{ij}=b_{ij}$ for $i<j$, $x_{ij}=0$ for $i>j$, and $x_{ii}=a_i/2$. Since $\bar{u}_1, \ldots, \bar{u}_r$ are linearly independent in $L_n/pL_n$, the map
$$
L_n/pL_n\to\F_p^r,\quad \bar{z}\mapsto (\langle\bar{u}_1,\bar{z}\rangle,\ldots,
\langle\bar{u}_r,\bar{z}\rangle)
$$
is surjective. Thus for each $1 \le j \le r$, there exists $\bar{z}_j\in L_n/pL_n$ such that $\langle\bar{u}_i,\bar{z}_j\rangle=x_{ij}$ ($1 \le i \le r$). Choose any lift $z_j\in L_n$ of $\bar{z}_j$. With these choices,
$$
u_i^{(m+1)}\equiv u_i^{(m)}+p^{\lambda_i+m}z_i \equiv u_i \pmod{p^{\lambda_i}}
$$
and $D_{ij}^{(m+1)} \in p^{\min(\lambda_i,\lambda_j)+m+1}\Z_p$, which completes the construction of $u_i^{(m+1)}$. 

Now for each $1 \le i \le r$, the sequence $(u_i^{(m)})_{m\ge0}$ converges to some $\wtil{u}_i\in L_n$. Since $D_{ij}^{(m)} \in p^{\min(\lambda_i, \lambda_j)+m}\Z_p$, $D_{ij}^{(m)}$ converges to $0$ as $m \to \infty$. Thus the vectors $\wtil{u}_1, \ldots, \wtil{u}_r\in L_n$ satisfy \eqref{eq4c}.
\end{proof}

Since the map $\Theta_n$ is invariant under the action of $\mO_n(\Z_p)$, it induces a map
\begin{equation} \label{eq4d}
\mO_n(\Z_p)\backslash\Sur(L_n,H) \to \Sym^2H.
\end{equation}

\begin{thm}\label{thm4e}
Let $H$ be a finite abelian $p$-group. Then for all sufficiently large $n$, \eqref{eq4d} is a bijection. Moreover, if $n>2d(H)$, then every $\mO_n(\Z_p)$-orbit in $\Sur(L_n,H)$ is a single $\SO_n(\Z_p)$-orbit.
\end{thm}

\begin{proof}
By Lemma \ref{lem4b}, the map is surjective for all sufficiently large $n$. Suppose that $\Theta_n(f)=\Theta_n(g)$ for $f,g \in \Sur(L_n,H)$ and choose frames $(u_1,\ldots,u_r)$ and $(v_1,\ldots,v_r)$ representing $f$ and $g$, respectively. By Lemmas \ref{lem4c} and \ref{lem4d}, we may replace each $u_i$ by an element of the coset $u_i+p^{\lambda_i}L_n$ such that $\langle u_i,u_j\rangle=\langle v_i,v_j\rangle$ for all $1 \le i,j \le r$. Since $\bar{u}_1, \ldots, \bar{u}_r$ and $\bar{v}_1, \ldots, \bar{v}_r$ are linearly independent, the $\Z_p$-submodules $M_u=\sum_{i=1}^r\Z_p u_i$ and $M_v=\sum_{i=1}^r\Z_p v_i$ are direct summands of $L_n$ so these are primitive sublattices.

By the relations $\langle u_i,u_j\rangle=\langle v_i,v_j\rangle$, the $\Z_p$-linear map $\alpha:M_u \to M_v$ ($u_i \mapsto v_i$) is an isometry with respect to the quadratic form $q(x)=\langle x,x\rangle$ on $L_n$. By James's extension theorem \cite[pp. 645--646]{Jam70}, $\alpha$ extends to an isometry of $L_n$ which is represented by a matrix $P \in\mO_n(\Z_p)$.
Set $Q=P^T=P^{-1}$. Then $Q^Tu_i=Pu_i=v_i$ for every $1\le i\le r$. By the description of frames under precomposition given above, $(v_1,\ldots,v_r)$ is a frame representing $f\circ Q$, which gives $g=f\circ Q$. Hence \eqref{eq4d} is injective, and therefore bijective.

Now assume that $n>2d(H)$. Fix $f\in\Sur(L_n,H)$, let $\bar{f}:L_n/pL_n\to H/pH$ be the reduction of $f$ modulo $p$ and set $W=\ker(\bar{f})$. Since $\bar{f}$ is surjective, we have
$$
\dim_{\F_p}W=n-d(H)>d(H)=\dim_{\F_p}W^\perp.
$$
Therefore $W$ is not a subset of $W^\perp$, so $\langle w,w \rangle\ne0$ for some $w\in W$. Indeed, if $\langle w,w \rangle=0$ for all $w \in W$, then 
$$
2\langle w_1,w_2 \rangle = \langle w_1+w_2,w_1+w_2 \rangle-\langle w_1,w_1 \rangle-\langle w_2,w_2 \rangle=0
$$
for all $w_1, w_2 \in W$ and $p$ is odd so $W \subseteq W^\perp$. Choose a lift $v_0\in L_n$ of $w$. Since $\bar{f}(w)=0$, we have $f(v_0)=ph$ for some $h\in H$. Choose $y\in L_n$ such that $f(y)=h$ and set $v=v_0-py$. Then $v\in\ker f$ and $\bar{v}=\bar{v_0}=w$, so $\langle v,v\rangle\in\Z_p^\times$.

The reflection $R_v : L_n \to L_n$ given by
$$
R_v(x)=x-2\frac{\langle x,v\rangle}{\langle v,v\rangle}v
$$
corresponds to a matrix in $\mO_n(\Z_p)$ with determinant $-1$. Since $v\in\ker f$, we have $f\circ R_v=f$. Thus the stabilizer of $f$ in $\mO_n(\Z_p)$ contains an element of determinant $-1$, so the $\mO_n(\Z_p)$-orbit of $f$ is equal to its $\SO_n(\Z_p)$-orbit.
\end{proof}

\begin{proof}[Proof of Theorem \ref{thm1a}(a)]
Choose a sufficiently large $n>2d(H)$ so that Theorem \ref{thm4e} applies. As in the proof of Proposition \ref{prop2b}, for every $A\in\mO_n(\Z_p)$, we have
$$
\#\Sur(\cok(A-I_n),H)=\#\{f\in\Sur(L_n,H):fA=f\}.
$$
By Theorem \ref{thm4e}, the action of $\mO_n(\Z_p)$ on $\Sur(L_n,H)$ has exactly $|\Sym^2H|$ orbits, and each $\mO_n(\Z_p)$-orbit is also a single $\SO_n(\Z_p)$-orbit. 

Let $G$ be either $\mO_n(\Z_p)$ or $\SO_n(\Z_p)$, and let $C_n$ be a Haar-random matrix in $G$. Let $\Omega_1, \ldots, \Omega_s$ ($s=|\Sym^2H|$) be the $\mO_n(\Z_p)$-orbits (equivalently, $\SO_n(\Z_p)$-orbits) in $\Sur(L_n,H)$. Applying Lemma \ref{lem4a} with $\Omega=\Omega_i$ and the trivial character $\psi=1$, we obtain
\begin{equation}\label{eq:cf1}
\bE(\# \{f\in \Omega_i: fC_n=f\})=\int_G \#\Fix_{\Omega_i}(g)\,dg=1.
\end{equation}
Therefore
\begin{equation}\label{eq:cf2}
\bE(\#\Sur(\cok(C_n-I_n),H))=\sum_{i=1}^{s} \bE(\# \{f\in \Omega_i: fC_n=f\})=|\Sym^2H|.    
\end{equation}
Since $\mO_n(\Z_p)=\SO_n(\Z_p)\sqcup\mO_n^-(\Z_p)$ and each component has Haar measure $1/2$, the same result holds when $C_n$ is a Haar-random matrix in $\mO_n^-(\Z_p)$.
\end{proof}

\section{Symplectic support and odd‑prime limiting laws} \label{Sec5}

The stable $H$-moments obtained in Theorem \ref{thm1a}(a) grow too rapidly to determine a unique distribution: distinct random finitely generated $\Z_p$-modules can have the same $H$-moments $|\Sym^2H|$ for every finite abelian $p$-group $H$. To determine the limiting distribution of $\cok(A_n-I_n)$, we need additional information on its support.
We follow the strategy of Nguyen and Wood \cite{NW25}. For random alternating matrices over $\Z_p$, they first identified the support of the cokernel and then combined this with the limiting $H$-moments $|\Sym^2H|$ to determine the limiting distribution using \cite[Theorem 4.1]{NW25}. We apply this theorem after identifying the support of $\cok(A_n-I_n)$.

First we construct a canonical perfect skew-symmetric pairing on the torsion subgroup of $\cok(A_n-I_n)$. The construction is valid for every prime $p$; when $p$ is odd, the pairing is alternating.
If $T$ is an endomorphism of a free $\Z_p$-module $L$ of finite rank and $V=L\otimes_{\Z_p}\Q_p$, then
\begin{equation}\label{eq5a}
\cok(T)_{\tors}=(L\cap T(V))/T(L).
\end{equation}
Indeed, the class of $x\in L$ in $L/T(L)$ is torsion if and only if $p^mx\in T(L)$ for some $m\ge0$, which is equivalent to $x\in T(V)$.

\begin{thm}\label{thm5a}
Let $L$ be a unimodular symmetric $\Z_p$-lattice and $A\in\mO(L)$. Then $\cok(A-I)_{\tors}$ admits a perfect skew-symmetric pairing
$$
\lambda_A:\cok(A-I)_{\tors}\times\cok(A-I)_{\tors} \rightarrow\Q_p/\Z_p.
$$
\end{thm}

\begin{proof}
Set $T=A-I$, $V=L\otimes_{\Z_p}\Q_p$, $K=\ker(T:V\to V)$ and $W=T(V)$. Let $T^*$ denote the adjoint of $T$ with respect to the bilinear form on $V$. Then $A \in \mO(L)$ so $A^*=A^{-1}$, which gives $T^*=A^{-1}-I=-A^{-1}T$. Hence $W=(\ker T^*)^\perp=K^\perp$, and \eqref{eq5a} implies that
$$
\cok(T)_{\tors}=(L\cap W)/T(L).
$$

For $x,y\in L\cap W$, choose $z\in V$ such that $Tz=x$ and define
$$
\lambda_A([x],[y]):=(\langle z,y\rangle \;\bmod{\Z_p}) \in \Q_p/\Z_p,
$$
where $[x]$ denotes the class of $x \in L \cap W$ in $(L\cap W)/T(L)$. We verify that the pairing $\lambda_A$ is well-defined. First, if $z'\in V$ satisfies $Tz'=x$, then $z'-z\in K$. Since $y\in W=K^\perp$, we have $\langle z'-z,y\rangle=0$ so the value is independent of the choice of $z$.
Next, if we replace $x$ by $x+T\ell$ ($\ell\in L$), then we may replace $z$ by $z+\ell$. Since
$$
\langle z+\ell,y\rangle-\langle z,y\rangle=\langle\ell,y\rangle\in\Z_p,
$$
the pairing is independent of the choice of $x$. Finally, replacing $y$ by $y+T\ell$ ($\ell\in L$) gives
$$
\langle z,y+T\ell\rangle-\langle z,y\rangle
=\langle z,T\ell\rangle=\langle T^*z,\ell\rangle=\langle -A^{-1}x,\ell\rangle \in \Z_p,
$$
so the pairing is independent of the choice of $y$.

The identity $(T^*+I)(T+I)=I$ gives $T^*+T=-T^*T$. For $x=Tz$ and $y=Tw$,
$$
\lambda_A([x],[y])+\lambda_A([y],[x])
=\langle z,Tw\rangle+\langle Tz,w\rangle \; \bmod{\Z_p}
$$
and
$$
\langle z,Tw\rangle+\langle Tz,w\rangle
=\langle(T^*+T)z,w\rangle=\langle -T^*Tz,w\rangle
=-\langle Tz,Tw\rangle=-\langle x,y\rangle \in \Z_p
$$
so $\lambda_A([x],[y])+\lambda_A([y],[x])=0$ in $\Q_p/\Z_p$. Hence $\lambda_A$ is skew-symmetric.

Finally, suppose that $[x]\in\cok(T)_{\tors}$ satisfies $\lambda_A([x],[y])=0$ for every $[y] \in\cok(T)_{\tors}$. Write $x=Tz$ with $z\in V$. By the definition of $\lambda_A$, we have $\langle z,L\cap W\rangle\subseteq\Z_p$. Since $W$ is a $\Q_p$-subspace of $V$, the submodule $L\cap W$ is saturated in $L$ so it is a direct summand of $L$. Therefore the $\Z_p$-linear map
$$
L \cap W \to \Z_p,\quad y\mapsto\langle z,y\rangle
$$
extends to $L$. Since $L$ is unimodular, there exists $\ell\in L$ such that
$$
\langle z-\ell,L\cap W\rangle=0.
$$
The $\Q_p$-span of $L\cap W$ is $W$, so $z-\ell\in W^\perp=K$. Hence $x=T\ell \in T(L)$ and $[x]=0$ in $\cok(T)_{\tors}$. 
We conclude that the pairing $\lambda_A$ is perfect.
\end{proof}

\begin{cor}\label{cor5b}
Let $p$ be an odd prime, $L$ be a unimodular symmetric $\Z_p$-lattice and $A\in\mO(L)$. Then
$$
\cok(A-I)_{\tors}\cong G\oplus G
$$
for some finite abelian $p$-group $G$.
\end{cor}

\begin{proof}
By Theorem \ref{thm5a}, $\cok(A-I)_{\tors}$ admits a perfect skew-symmetric pairing $\lambda_A$. Since $p$ is odd, $\lambda_A$ is alternating. Thus $\cok(A-I)_{\tors}$ is a symplectic group, so it is square.
\end{proof}

\begin{proof}[Proof of Theorem \ref{thm1b}]
By Proposition \ref{prop3a} and Corollary \ref{cor5b}, we have
\begin{equation}\label{eq5b}
\cok(A_n-I_n)\cong
\begin{cases}
G\oplus G, & \det A_n=(-1)^n,\\
\Z_p\oplus G\oplus G, & \det A_n=(-1)^{n+1}
\end{cases}
\end{equation}
for some finite abelian $p$-group $G$ almost surely.
First we assume that $\det A_n=(-1)^n$. By Theorem \ref{thm1a}(a), for every finite abelian $p$-group $H$ and all sufficiently large $n$, we have
$$
\bE(\#\Sur(\cok(A_n-I_n),H))=|\Sym^2H|.
$$
Then \cite[Theorem 4.1]{NW25}, applied to \eqref{eq5b}, implies that $\cok(A_n-I_n)\dist\mu_{p,0}$.

Now we assume that $\det A_n=(-1)^{n+1}$. By \eqref{eq5b}, for every positive integer $k$, the reduction of $\cok(A_n-I_n)$ modulo $p^k$ is of the form
$$
\Z/p^k\Z\oplus G/p^kG\oplus G/p^kG,
$$
which is precisely the form required in condition (2) of \cite[Theorem 4.1]{NW25}. Together with the $H$-moments in Theorem \ref{thm1a}(a), the argument in the proof of \cite[Theorem 1.13]{NW25} for the odd-dimensional case can be directly applied. Thus $\cok(A_n-I_n)$ has the same limiting distribution as $\cok(B_{2g+1})$ for a Haar-random matrix $B_{2g+1} \in \Alt_{2g+1}(\Z_p)$, so $\cok(A_n-I_n)\dist\mu_{p,1}$.

Finally, $\mO_n(\Z_p)=\SO_n(\Z_p)\sqcup\mO_n^-(\Z_p)$ and each component has Haar measure $1/2$, so the limiting law of the full orthogonal group $\mO_n(\Z_p)$ is $(1/2)\mu_{p,0}+(1/2)\mu_{p,1}$.
\end{proof}


\section{Dyadic orbit invariants}\label{Sec6}

For the rest of the paper, we fix $p=2$. 
Let
$$
H=\bigoplus_{i=1}^r\Z/2^{\lambda_i}\Z\,h_i,
\quad
\lambda_1\ge\cdots\ge\lambda_r\ge1.
$$
Let $(u_1,\ldots,u_r)$ be a frame representing $f$, so that
$$
f(x)=\sum_{i=1}^r
(\langle u_i,x\rangle \bmod 2^{\lambda_i})h_i.
$$
The function $\Theta_n$ introduced in Section \ref{Sec4} is given by
$$
\Theta_n(f)=\sum_{k=1}^n f(e_k)^2 = \sum_{i=1}^r
\langle u_i,u_i\rangle h_i^2
+
2\sum_{1\le i<j\le r}
\langle u_i,u_j\rangle h_i h_j \in \Sym^2H.
$$
We remark that $\Theta_n(f)$ no longer records enough information when $p=2$. Since the order of $h_i^2 \in \Sym^2H$ is $2^{\lambda_i}$, $\Theta_n(f)$ records the
diagonal Gram entry $\langle u_i,u_i\rangle$ only modulo
$2^{\lambda_i}$.
Although the frame vector $u_i$ itself is only determined modulo
$2^{\lambda_i}L_n$, the quantity
$$
\langle u_i,u_i\rangle \pmod{2^{\lambda_i+1}}
$$
is nevertheless well-defined by the homomorphism $f$. Indeed, replacing
$u_i$ by another representative
$$
u_i'=u_i+2^{\lambda_i}z_i
$$
does not change $f$, while
\begin{align*}
\langle u_i',u_i'\rangle-\langle u_i,u_i\rangle
&=
2^{\lambda_i+1}\langle u_i,z_i\rangle
+
2^{2\lambda_i}\langle z_i,z_i\rangle.
\end{align*}
Thus
$$
\langle u_i',u_i'\rangle
\equiv
\langle u_i,u_i\rangle
\pmod{2^{\lambda_i+1}}.
$$
Consequently, there is one additional $2$-adic digit of diagonal Gram data attached intrinsically to $f$ which is invisible in $\Theta_n(f)\in \Sym^2H$. Whitehead's universal quadratic group provides the appropriate refinement that retains this additional information.

\subsection{Whitehead’s universal quadratic group} \label{Sub61}
For abelian groups $H$ and $A$, a map $q: H \to A$ is said to be \emph{homogeneous quadratic} if the map
$$
B_q(x,y) = q(x+y) - q(x)-q(y)
$$
is biadditive and $q(kx) = k^2q(x)$ for all $k \in \Z$. 
\begin{defn}
Let $H$ be an abelian group. The \emph{Whitehead group} $\Gamma(H)$ is an abelian group equipped with a universal homogeneous quadratic map
$$
\gamma_H : H \to \Gamma(H)
$$
with the property that for every homogeneous quadratic map $q: H \to A$, there is a unique group homomorphism $\tilde{q}: \Gamma(H) \to A$ satisfying $q= \tilde{q} \circ \gamma_H$. We often abbreviate $\gamma = \gamma_H$ and write
$$
[x,y] = \gamma(x+y) - \gamma(x) - \gamma(y).
$$
The map $(x,y) \mapsto [x,y]$ is then symmetric and biadditive.
\end{defn}
The universal quadratic map $\gamma$ satisfies
\begin{equation}
\label{eq:universal quadratic map}
\gamma(x+y)=\gamma(x)+\gamma(y)+[x,y],\quad
[ax,by]=ab[x,y],\quad
\gamma(ax)=a^2\gamma(x)
\end{equation}
for $a, b \in \Z$. 

We recall the following standard properties of Whitehead's group (see \cite[pp. 16--17]{Bau96}).
For $m \ge 1$,
$$
\Gamma(\Z/2^m\Z) \cong \Z/2^{m+1}\Z,
$$
and for abelian groups $A$ and $B$, there is a canonical decomposition
$$
\Gamma(A \oplus B) \cong \Gamma(A) \oplus \Gamma(B) \oplus (A \otimes B).
$$
Under this decomposition, the maps from $\Gamma(A)$ and $\Gamma(B)$ to $\Gamma(A \oplus B)$ are induced by the inclusions $A, B \hookrightarrow A \oplus B$, while the $A \otimes B$-summand maps via $a \otimes b \mapsto [(a,0),(0,b)]$. 
Hence, for $$
H = \bigoplus_{i=1}^r \Z/2^{\lambda_i}\Z h_i, \quad \lambda_1 \ge \cdots \ge \lambda_r \ge 1,
$$
we obtain the decomposition
\begin{equation}
\label{eq:gamma-decomposition}
\Gamma(H) \cong \bigoplus_{i=1}^r \Z/2^{\lambda_i +1}\Z \gamma(h_i) \oplus \bigoplus_{i < j} \Z/2^{\min(\lambda_i, \lambda_j)}\Z [h_i, h_j].
\end{equation}
Therefore, it follows that
\begin{equation} \label{eq:gamma-size}
|\Gamma(H)| = 2^r|\Sym^2H|.
\end{equation}


\subsection{The characteristic functional and quadratic Gram invariants}
Let $e_1, \ldots, e_n$ denote the standard basis of $L_n$, and put
$$
\bar{L}_n=L_n/2L_n,\quad
c_n=e_1+\cdots+e_n\in L_n,\quad
\bar{c}_n=c_n\bmod 2L_n\in\bar{L}_n.
$$
For $x = (x_1, \ldots, x_n) \in L_n$, we write $\bar{x}$ for its image in $\bar{L}_n$. The standard bilinear form on $\bar L_n$ is nondegenerate.
By abuse of notation, we also denote it by $\langle\cdot,\cdot\rangle$.
For every $\bar x = (\bar x_1, \ldots, \bar x_n)\in\bar L_n$,
\begin{equation} \label{eq:characteristic-class}
\langle \bar{c}_n, \bar{x}\rangle = \sum_{i=1}^n \bar{x}_i = \sum_{i=1}^n \bar{x}_i^2 = \langle \bar{x}, \bar{x}\rangle.
\end{equation}
In other words, $\bar{c}_n$ represents the functional $\bar{x} \mapsto \langle \bar{x}, \bar{x}\rangle$. For every $U\in\mO_n(\Z_2)$, its reduction
$\bar U\in\mO_n(\F_2)$ fixes $\bar c_n$.

For a surjection $f: L_n \twoheadrightarrow H$, define
$$
q_f:= \sum_{i=1}^n \gamma(f(e_i)) \in \Gamma(H).
$$
Reduction modulo $2$ induces a surjection
$$
\bar{f}: \bar{L}_n \twoheadrightarrow H/2H
$$
and hence an injection
$$
\bar{f}^*: \Hom(H, \F_2) \cong \Hom(H/2H, \F_2) \hookrightarrow \Hom(\bar{L}_n, \F_2).
$$
Let $(u_i)$ be a frame representing $f$. For $\chi \in \Hom(H, \F_2)$, let $\epsilon_i = \chi(h_i)$. Then
$$
(\bar{f}^*\chi)(\bar{x}) = \chi(\bar{f}(\bar{x})) = \sum_{i=1}^r \epsilon_i \langle \bar{u}_i, \bar{x}\rangle = \left\langle \sum_{i=1}^r \epsilon_i \bar{u}_i, \bar{x}\right\rangle.
$$
Identify $\bar{L}_n$ with $\Hom(\bar{L}_n, \F_2)$ via the standard form.
By the definitions of $\bar{f}^*$ and $\bar{c}_n$, and the fact that the standard form is nondegenerate, we immediately obtain the following.
\begin{lem}
\label{lem:barf image equivalent condition}
For $f \in \Sur(L_n, H)$, the following statements are equivalent.
\begin{enumerate}[label=(\arabic*)]
\item
$\bar{c}_n \in \im(\bar{f}^*)$
\item 
There exists $\chi_f \in \Hom(H, \F_2)$ such that 
$$\chi_f(f(x)) \equiv \langle x,x \rangle \pmod{2}$$ for all $x \in L_n$.
\end{enumerate}
Moreover, if these conditions hold, $\chi_f$ is unique because $\bar{f}^*$ is injective, and $\chi_f \neq 0$ since $\bar{c}_n \neq 0$.
\end{lem}

\begin{defn}
If the equivalent conditions in Lemma \ref{lem:barf image equivalent condition} hold for $f \in \Sur(L_n, H)$, we say that $f$ is \emph{characteristic} and call $\chi_f$ its \emph{characteristic functional}. Otherwise, we say that $f$ is \emph{noncharacteristic}. 
\end{defn}

\begin{lem}
\label{lem:Dyadic invariance}
Let $f \in \Sur(L_n, H)$. If $U \in \mO_n(\Z_2)$, then 
$$
q_{f\circ U} = q_f.
$$
Moreover, $f \circ U$ is characteristic if and only if $f$ is characteristic, and in that case $\chi_{f\circ U} = \chi_f.$
\end{lem}

\begin{proof}
By \eqref{eq:universal quadratic map}, it follows that for $x_1, \ldots, x_m \in H$, the identity
\begin{equation}\label{eq:universal quadratic expansion}
\gamma\left(\sum_i a_ix_i\right)
=\sum_i a_i^2\gamma(x_i)
+\sum_{i<j}a_ia_j[x_i,x_j]
\end{equation}
holds for all $a_1, \ldots, a_m \in \Z_2$. Now we put $U = (u_{ij})$ and $x_i = f(e_i)$. Then \eqref{eq:universal quadratic expansion} proves that
\begin{align*}
q_{f \circ U} & = \sum_{j} \gamma \left(\sum_i u_{ij}x_i\right) \\
& = \sum_i\left(\sum_j u_{ij}^2\right) \gamma(x_i) + \sum_{i < j}\left(\sum_k u_{ik} u_{jk} \right)[x_i, x_j] \\
& = q_f.
\end{align*}
Moreover, if $f$ is characteristic, then
$$
\chi_f(f(Ux))
\equiv \langle Ux,Ux\rangle
=\langle x,x\rangle \pmod 2.
$$
Thus $f\circ U$ is characteristic with characteristic functional $\chi_f$, so $\chi_{f\circ U}=\chi_f$. The converse follows similarly.
\end{proof}

Now represent $f$ by a frame $(u_i)$, and let $F \in \M_{r \times n}(\Z_2)$ be its row-frame matrix whose $i$-th row is $u_i^T$. Since $f$ is surjective, the vectors $\bar{u}_1,\ldots, \bar{u}_r \in \bar{L}_n$ are linearly independent. Let 
$$
J_f=\sum_{i=1}^r\Z_2u_i\subset L_n
$$
and
$$
\bar J_f=\mathrm{span}_{\F_2}
\{\bar u_1,\ldots,\bar u_r\}\subset\bar L_n.
$$
A sublattice $J\subset L_n$ is called \emph{primitive} if $L_n/J$ is torsion-free, equivalently if $J$ is a direct summand. Since the reduction of $F$ has rank $r$, some $r\times r$ minor of $F$ is a unit. Smith normal form therefore shows that $J_f$ is primitive. The element $q_f\in\Gamma(H)$ can be read directly from the Gram matrix $FF^{T}$. 

\begin{lem}\label{lem:dyadic-gram}
Let $f$ and $g$ be represented by frames $(u_i)$ and $(v_i)$, respectively. Then $q_f=q_g$ if and only if
\begin{align}
\langle u_i,u_i\rangle
&\equiv\langle v_i,v_i\rangle
\pmod{2^{\lambda_i+1}},
\label{eq:dyadic-diagonal}\\
\langle u_i,u_j\rangle
&\equiv\langle v_i,v_j\rangle
\pmod{2^{\min(\lambda_i,\lambda_j)}}
\quad(i\ne j).
\label{eq:dyadic-offdiag}
\end{align}
\end{lem}

\begin{proof}
Writing $u_{ik}=\langle u_i,e_k\rangle$, we have $f(e_k)=\sum_i u_{ik}h_i$. By \eqref{eq:universal quadratic map} and using $\sum_k u_{ik}u_{jk}=\langle u_i,u_j\rangle$, we have
\begin{equation}\label{eq:qf-gram-expansion}
q_f=
\sum_i\langle u_i,u_i\rangle\gamma(h_i)
+\sum_{i<j}\langle u_i,u_j\rangle[h_i,h_j].
\end{equation}
In the decomposition \eqref{eq:gamma-decomposition}, the element $\gamma(h_i)$ has order $2^{\lambda_i+1}$, whereas $[h_i,h_j]$ has order $2^{\min(\lambda_i,\lambda_j)}$. Equality of the corresponding coefficients in \eqref{eq:qf-gram-expansion} is therefore exactly the pair of congruences \eqref{eq:dyadic-diagonal} and \eqref{eq:dyadic-offdiag}. 
\end{proof}

\begin{rmk}
In Lemma \ref{lem:dyadic-gram}, the diagonal modulus is one power of $2$ larger than the modulus of $h_i^2$ in $\Sym^2H$; this is the extra information carried by $\Gamma(H)$.
\end{rmk}

\subsection{Exact correction of the Gram matrix}

The congruences in Lemma \ref{lem:dyadic-gram} are weaker than equality of Gram matrices. The next lemma shows that, after changing each frame vector within the congruence class that represents the same homomorphism to $H$, the Gram matrices can be made exactly equal. 

\begin{lem}
\label{lem:weighted-gram}
Let $(u_1,\ldots,u_r)$ and $(v_1,\ldots,v_r)$ be two $r$-tuples in $L_n$ whose reductions modulo $2$ are linearly independent, and put
$$
\bar J_u=\mathrm{span}_{\F_2}\{\bar u_1,\ldots,\bar u_r\},
\quad
\bar J_v=\mathrm{span}_{\F_2}\{\bar v_1,\ldots,\bar v_r\}.
$$
Assume \eqref{eq:dyadic-diagonal} and \eqref{eq:dyadic-offdiag}. Assume also that either
\begin{enumerate}[label=(\roman*)]
\item $\bar c_n\notin\bar J_u\cup\bar J_v$; or
\item $\bar c_n$ lies in both reduced row spans with the same coefficient vector:
$$
\bar c_n=\sum_{i=1}^r\epsilon_i\bar u_i
=\sum_{i=1}^r\epsilon_i\bar v_i,
\quad \epsilon_i\in\F_2.
$$
\end{enumerate}
Then there are vectors
$$
\widetilde u_i\equiv u_i\pmod{2^{\lambda_i}L_n}
$$
such that
$$
\langle\widetilde u_i,\widetilde u_j\rangle
=\langle v_i,v_j\rangle
\quad(1\le i,j\le r).
$$
\end{lem}

\begin{proof}
We construct $\widetilde u_1,\ldots,\widetilde u_r$ in this order. Suppose that $\widetilde u_i$ have been chosen for $i<j$ so that
\begin{equation*}
\widetilde u_i\equiv u_i\pmod{2^{\lambda_i}L_n},
\quad
\langle\widetilde u_i,\widetilde u_k\rangle
=\langle v_i,v_k\rangle\quad(i,k<j).
\end{equation*}
Set $P=\sum_{i<j}\Z_2\widetilde u_i$, and let $\bar P$ be its image in $\bar L_n$. Since the reductions of the $\widetilde u_i$ are linearly independent, the map
\begin{equation*}
L_n\rightarrow\Z_2^{j-1},\quad
z\mapsto(\langle\widetilde u_i,z\rangle)_{i<j}
\end{equation*}
is surjective. For $i<j$, the inequality $\lambda_i\ge\lambda_j$ and \eqref{eq:dyadic-offdiag} give
\begin{equation*}
\langle v_i,v_j\rangle-\langle\widetilde u_i,u_j\rangle
\in2^{\lambda_j}\Z_2.
\end{equation*}
Thus we can choose $z_0\in L_n$ such that
\begin{equation*}
y=u_j+2^{\lambda_j}z_0,\quad
\langle\widetilde u_i,y\rangle=\langle v_i,v_j\rangle\quad(i<j).
\end{equation*}
Since $\langle y,y\rangle\equiv\langle u_j,u_j\rangle\pmod{2^{\lambda_j+1}}$, equation \eqref{eq:dyadic-diagonal} implies that
\begin{equation*}
a=\frac{\langle v_j,v_j\rangle-\langle y,y\rangle}{2^{\lambda_j+1}}
\in\Z_2.
\end{equation*}

Set $M=\{z\in L_n:\langle P,z\rangle=0\}$. Surjectivity of the preceding map implies that the image of $M$ in $\bar L_n$ is $\bar P^\perp$. 
Indeed, if $\bar z_0\in\bar P^\perp$ and $z_0\in L_n$ is a lift, then
$\langle\widetilde u_i,z_0\rangle=2a_i$ for some $a_i\in\Z_2$.
By surjectivity, choose $w\in L_n$ with
$\langle\widetilde u_i,w\rangle=-a_i$ for all $i<j$.
Then $z_0+2w\in M$ and has reduction $\bar z_0$. The reverse inclusion is immediate.

Since $\bar y=\bar u_j\notin\bar P$, we can choose $z\in M$ with $\langle y,z\rangle$ odd. For such a $z$, consider
$$
\widetilde u_j=y+2^{\lambda_j}tz,
\quad t\in\Z_2,
$$
where $t$ will be chosen below. Then $\langle \widetilde u_i, \widetilde u_j\rangle = \langle v_i, v_j\rangle$ for $i<j$. Moreover, 
$$\langle\widetilde u_j, \widetilde u_j \rangle = \langle v_j, v_j \rangle ~\Longleftrightarrow ~0 = \langle y,z\rangle t+2^{\lambda_j-1}\langle z,z\rangle t^2-a=:F(t).$$ 
For every $t\in\Z_2$,
\begin{equation*}
F'(t)=\langle y,z\rangle+2^{\lambda_j}\langle z,z\rangle t\in\Z_2^\times.
\end{equation*}
It therefore suffices, by Hensel's lemma, to choose $z$ so that $F$ has a root modulo $2$.

If $\lambda_j\ge2$, any such $z$ works, since $F(t)\equiv t-a\pmod2$. Suppose that $\lambda_j=1$ and $\bar c_n\notin\bar y+\bar P$. We can then choose $z\in M$ satisfying
\begin{equation*}
\langle y,z\rangle\equiv1,\quad
\langle c_n,z\rangle\equiv0\pmod2.
\end{equation*}
Indeed, if $\bar c_n\in\bar P$, the second condition is automatic; otherwise the images of $\bar y$ and $\bar c_n$ in $\bar L_n/\bar P$ are linearly independent, so their pairings on $\bar P^\perp$ can be prescribed independently. By \eqref{eq:characteristic-class}, $\langle z,z\rangle$ is even, and again $F(t)\equiv t-a\pmod2$.

It remains to consider $\lambda_j=1$ and $\bar c_n\in\bar y+\bar P$. This can occur only in case (ii). Uniqueness of the coefficients in the independent frame gives $\epsilon_j=1$ and $\epsilon_i=0$ for $i>j$. Regard the $\epsilon_i$ as elements of $\{0,1\}\subset\Z_2$, and set
\begin{equation*}
w_y=y+\sum_{i<j}\epsilon_i\widetilde u_i,
\quad
w_v=v_j+\sum_{i<j}\epsilon_i v_i.
\end{equation*}
Both vectors reduce to $\bar c_n$, so all their coordinates are odd and
\begin{equation*}
\langle w_y,w_y\rangle\equiv\langle w_v,w_v\rangle\equiv n\pmod8.
\end{equation*}
The pairings already fixed give
\begin{equation*}
\langle v_j,v_j\rangle-\langle y,y\rangle
=\langle w_v,w_v\rangle-\langle w_y,w_y\rangle\in8\Z_2.
\end{equation*}
Hence $a\in2\Z_2$, and any $z\in M$ with $\langle y,z\rangle$ odd gives $F(0)\equiv0\pmod2$.

In every case, Hensel's lemma gives a root $t\in\Z_2$ of $F$. The resulting vector $\widetilde u_j=y+2^{\lambda_j}tz$ satisfies $\widetilde u_j\equiv u_j\pmod{2^{\lambda_j}L_n}$ and all the required pairings with $\widetilde u_i$ for $i\le j$. This completes the induction.
\end{proof}


\subsection{James's extension theorem and orbit classification}\label{Sub64}
A primitive vector $\alpha_0\in L_n$ (that is, $\alpha_0\notin 2L_n$) is called \emph{characteristic} if $\bar\alpha_0=\bar c_n$, equivalently if
$$
\langle\alpha_0,x\rangle\equiv\langle x,x\rangle\pmod2
$$
for every $x \in L_n$. For $0\ne\alpha\in L_n$, write
$$
\alpha=2^{s(\alpha)}\alpha_0,
\quad
s(\alpha)=\max\{s:\alpha\in2^sL_n\},
$$
with $\alpha_0$ primitive. James's invariant \cite[p. 645]{Jam70} is given by
$$
T(\alpha)=
\begin{cases}
1 &\bar\alpha_0=\bar c_n,\\
0 &\bar\alpha_0\ne\bar c_n,
\end{cases}
\quad T(0)=0.
$$
Recall that a sublattice $J\subseteq L_n$ is primitive if
$L_n/J$ is torsion-free, equivalently, if $J$ is a direct summand of $L_n$.
\begin{lem}\label{lem:james-tailored}
Let $J,K\subset L_n$ be primitive sublattices and let $\varphi:J\to K$ be an isometry. Then $\varphi$ extends to an element of $\mO_n(\Z_2)$ if and only if, for every primitive $\alpha\in J$,
$$
\bar\alpha=\bar c_n
\quad\Longleftrightarrow\quad
\overline{\varphi(\alpha)}=\bar c_n.
$$
\end{lem}

\begin{proof}
James's extension theorem \cite[pp. 645--646]{Jam70} says that $\varphi$ extends if and only if
$$
T(\alpha)=T(\varphi(\alpha)) \text{ for all $\alpha \in J$}.
$$
Since $J$ and $K$ are primitive sublattices,
\begin{equation*}
J\cap2^sL_n=2^sJ
\quad\text{and}\quad
K\cap2^sL_n=2^sK.
\end{equation*}
Thus, for $0\ne\alpha=2^{s(\alpha)}\alpha_0\in J$, one has
$$
\varphi(\alpha)=2^{s(\alpha)}\varphi(\alpha_0)
$$
with $\varphi(\alpha_0)$ primitive. Hence the assertion follows from the definition of $T$.
\end{proof}

\begin{thm}\label{thm:dyadic-orbits}
Let $H$ be a finite abelian $2$-group. For $n \ge d(H)$, two surjections $f,g:L_n\twoheadrightarrow H$ lie in the same $\mO_n(\Z_2)$-orbit if and only if
\begin{enumerate}[label=(\roman*)]
\item $q_f=q_g$ in $\Gamma(H)$; and
\item either both maps are noncharacteristic, or both are characteristic and $\chi_f=\chi_g$.
\end{enumerate}
\end{thm}

\begin{proof}
Necessity is Lemma \ref{lem:Dyadic invariance}. For sufficiency, choose frames $(u_i)$ and $(v_i)$ representing $f$ and $g$, respectively. Let $F$ and $G$ be their corresponding row-frame matrices. Equality $q_f=q_g$ gives \eqref{eq:dyadic-diagonal} and \eqref{eq:dyadic-offdiag} by Lemma \ref{lem:dyadic-gram}. If the maps are characteristic with common functional $\chi$, put $\epsilon_i=\chi(h_i)$. Then
$$
\bar c_n=\sum_i\epsilon_i\bar u_i
=\sum_i\epsilon_i\bar v_i.
$$
If they are noncharacteristic, $\bar c_n \not\in \bar J_u\cup\bar J_v$, where $\bar J_u$ and $\bar J_v$ are as in Lemma \ref{lem:weighted-gram}.
Then Lemma \ref{lem:weighted-gram} allows us, without changing $f$, to assume
$$
\langle u_i,u_j\rangle=\langle v_i,v_j\rangle
\quad(1\le i,j\le r).
$$
Since $f$ and $g$ are surjective, the reductions of the $u_i$ and of the $v_i$ are linearly
independent. Therefore, both $F$ and $G$ have an $r\times r$ minor with unit determinant.
By Smith normal form,
$$
J=\sum_i\Z_2u_i \quad\text{and}\quad K=\sum_i\Z_2v_i
$$
are primitive sublattices of $L_n$. Moreover, $u_i\mapsto v_i$ defines an isometry
$$
\varphi:J \to K.
$$
In the noncharacteristic case, neither $J/2J$ nor $K/2K$ contains $\bar c_n$. In the characteristic case, a primitive vector $\sum_i a_iu_i$ reduces to $\bar c_n$ exactly when $\bar a_i = \epsilon_i$ for all $1\le i \le r$, and the same criterion holds for $\sum_i a_iv_i$. Lemma \ref{lem:james-tailored} then extends $\varphi$ to some $V\in\mO_n(\Z_2)$. With $U=V^{T}$, we obtain $FU=G$, so $g=f\circ U$.
\end{proof}


\section{Counting dyadic orbits}\label{Sec7}
Retain the fixed decomposition
$$
H=\bigoplus_{i=1}^r\Z/2^{\lambda_i}\Z\,h_i,
\quad \lambda_1\ge\cdots\ge\lambda_r\ge1,
$$
from Section \ref{Sec6}. For a nonzero functional $\chi\in\Hom(H,\F_2)$, put
$$
C_\chi=\{x\in H:\chi(x)=1\}.
$$
Choose $x_\chi\in C_\chi$ and define
$$
D_\chi=
\angles{\gamma(x)-\gamma(y):x,y\in C_\chi} \subset\Gamma(H).
$$
It is immediate that the coset $n\gamma(x_\chi)+D_\chi$ is independent of $x_\chi$.

\subsection{Surjections with a fixed characteristic functional}\label{subsec:characteristic-strata}

We begin with an elementary fact about sufficiently long sums in a finite abelian group.

\begin{lem}\label{lem:stable-sumset}
Let $X$ be a finite set, let $A$ be a finite abelian group, let $\sigma:X\to A$, and choose $x_0\in X$. Put
$$
D=\angles{\sigma(x)-\sigma(x_0):x\in X}\subset A.
$$
Fix $y_1,\ldots,y_\ell\in X$. For all sufficiently large $n$,
$$
\left\{\sum_{i=1}^{\ell}\sigma(y_i)
+\sum_{j=\ell+1}^{n}\sigma(x_j):x_j\in X\right\}
=n\sigma(x_0)+D.
$$
\end{lem}

\begin{proof}
Set
\begin{equation*}
S=\{\sigma(x)-\sigma(x_0):x\in X\}.
\end{equation*}
Since $0\in S$ and $S$ generates the finite group $D$, there exists
$N$ such that every element of $D$ is a sum of at most $N$ elements
of $S$. Padding with zeros shows that every element of $D$ is a
sum of exactly $m$ elements of $S$ for every $m\ge N$.
Since
\begin{equation*}
\sum_{i=1}^{\ell}(\sigma(y_i)-\sigma(x_0))\in D,
\end{equation*}
the asserted equality follows for $n\ge N+\ell$.
\end{proof}

For $n\ge1$ and $0\neq \chi \in \Hom(H, \F_2)$, define
$$
\mc Q_{n,\chi}
:=\{q_f:f\in\Sur(L_n,H), ~ f\text{ is characteristic and}\ \chi_f=\chi\}
$$
and
$$
\mc Q_n^{\mathrm{nc}}
:=\{q_f:f\in\Sur(L_n,H)\text{ and } f\text{ is noncharacteristic}\}.
$$

\begin{prop}
\label{prop:stable-dyadic-images}
Let $0 \neq \chi \in\Hom(H,\F_2)$. For all sufficiently large $n$,
\begin{equation}\label{eq:stable-characteristic-image}
\mc Q_{n,\chi}=n\gamma(x_\chi)+D_\chi
\end{equation}
and
\begin{equation}\label{eq:stable-noncharacteristic-image}
\mc Q_n^{\mathrm{nc}}=\Gamma(H).
\end{equation}
\end{prop}

\begin{proof}
If $f$ is characteristic and $\chi_f=\chi$, then 
$$
\chi(f(e_k)) = \langle e_k, e_k\rangle = 1 ~\text{in $\F_2$},
$$
so $f(e_k) \in C_\chi$ for every $1 \le k \le n$. Therefore,
$$
q_f-n\gamma(x_\chi)
=\sum_{k=1}^n(\gamma(f(e_k))-\gamma(x_\chi))
\in D_\chi.
$$
This proves one inclusion in \eqref{eq:stable-characteristic-image}.

For the reverse inclusion, choose $x_0 \in C_\chi$ and generators $k_1, \ldots, k_s$ of $\ker \chi$. 
Let $q \in n\gamma(x_\chi) + D_\chi$. Apply Lemma \ref{lem:stable-sumset} with
$$
X=C_\chi,
\quad A=\Gamma(H),
\quad \sigma=\gamma,
\quad x_0=x_\chi.
$$
Since
$$
\left\langle \gamma(x)-\gamma(x_\chi): x\in C_\chi\right\rangle=D_\chi,
$$
for all sufficiently large $n$, one can write
$$
q=\sum_{k=1}^n\gamma(y_k),
$$
where
$$
y_1=x_0,\quad
y_{i+1}=x_0+k_i\quad(1\le i\le s),
$$
and
$$
y_k\in C_\chi\quad(1\le k\le n).
$$
Now take $f \in \Hom(L_n, H)$ so that $f(e_k) = y_k$ for all $1 \le k \le n$. Since $y_1, \ldots, y_{s+1}$ generate $H$, $f$ is surjective. Furthermore, if $x = \sum_{k=1}^n a_ke_k$ with $a_k \in \Z_2$, we have
$$
\chi(f(x)) = \sum_k \bar{a}_k\chi(y_k) = \sum_{k} \bar{a}_k = \sum_k \bar{a}_k^2 = \langle x, x\rangle \pmod{2}.
$$
It follows that $f$ is characteristic, $\chi = \chi_f$, and $q = q_f$. This proves \eqref{eq:stable-characteristic-image}.

For \eqref{eq:stable-noncharacteristic-image}, let $q' \in \Gamma(H)$. Apply Lemma \ref{lem:stable-sumset} with
$$
X=H,
\quad A=\Gamma(H),
\quad \sigma=\gamma,
\quad x_0=0.
$$
By \eqref{eq:gamma-decomposition}, $\Gamma(H)$ is generated by the image of $\gamma: H \to \Gamma(H)$.
As above, for all sufficiently large $n$, one can write
$$
q'=\sum_{k=1}^n\gamma(z_k),
$$
where
$$
z_1=0,\quad
z_{i+1}=h_i\quad(1\le i\le r),
$$
and
$$
z_k\in H\quad(1\le k\le n).
$$
Now choose $f' \in \Hom(L_n, H)$ such that $f'(e_k) = z_k$ for all $1 \le k \le n$. Then $f'$ is surjective and $q'= q_{f'}$. If $f'$ were characteristic, then it would satisfy
$$
\chi_{f'}(f'(e_1)) = \langle e_1, e_1\rangle = 1~\text{in $\F_2$},
$$
which is a contradiction to our assumption that $f'(e_1) = 0$. This completes the proof of \eqref{eq:stable-noncharacteristic-image}.
\end{proof} 

Combining Proposition \ref{prop:stable-dyadic-images} with Theorem \ref{thm:dyadic-orbits} gives, for all sufficiently large $n$,
\begin{equation}\label{eq:dyadic-orbit-bijection}
\mO_n(\Z_2)\backslash\Sur(L_n,H)
\cong
\{\mathrm{nc}\}\times\Gamma(H)
\ \sqcup\
\left\{(\chi,q):
\begin{array}{l}
0\ne\chi\in\Hom(H,\F_2),\\
q\in n\gamma(x_\chi)+D_\chi
\end{array}\right\}.
\end{equation}
Here $\mathrm{nc}$ abbreviates noncharacteristic and labels the orbits represented by surjections $f$ that are not characteristic. Therefore, for all sufficiently large $n$,
\begin{equation}\label{eq:number-dyadic-orbits}
\#(\mO_n(\Z_2)\backslash\Sur(L_n,H))
=|\Gamma(H)|+\sum_{0\ne\chi}|D_\chi|.
\end{equation}

\subsection{The index of \texorpdfstring{$D_\chi$}{Dchi} in \texorpdfstring{$\Gamma(H)$}{Gamma(H)}} \label{Sub72}
\begin{lem}
\label{lem:Dchi-index}
Let $r=d(H)$. Then
$$
[\Gamma(H):D_\chi]=
\begin{cases}
2^{r+1},&\chi|_{H[2]}\ne0,\\
2^{r+2},&\chi|_{H[2]}=0.
\end{cases}
$$
\end{lem}

\begin{proof}
Set $K=\ker\chi$, choose $x_0\in C_\chi$, and write $t=2x_0\in K$ and $\bar K=K/2K$. Let $Q=\Gamma(H)/D_\chi$, and let $q:H\to Q$ be the homogeneous quadratic map induced by $\gamma: H \to \Gamma(H)$. Since $q(x_0+k)=q(x_0)$ for $k\in K$, its associated bilinear form satisfies
\begin{equation*}
q(k)=-B_q(x_0,k).
\end{equation*}
Thus $\ell=q|_K$ is additive. Homogeneity then gives for $k \in K$
\begin{equation*}
2\ell(k)=\ell(2k)=4\ell(k),\quad
4q(x_0)=\ell(t),\quad 8q(x_0)=0.
\end{equation*}
It follows from \eqref{eq:gamma-decomposition} that $Q$ is generated by $q(H)$, so these identities define a surjective homomorphism
\begin{equation*}
E:=((\Z/8\Z)a\oplus\bar K)/\angles{4a-\bar t}
\rightarrow Q,\quad
a\mapsto q(x_0),\quad\bar k\mapsto\ell(k).
\end{equation*}
We show that this map is an isomorphism by constructing its inverse.

The map $(n,k)\mapsto nx_0+k$ induces an isomorphism
\begin{equation*}
H\cong(\Z\oplus K)/\angles{(2,-t)}.
\end{equation*}
Define $f:\Z\oplus K\to E$ by $f(n,k)=n^2a+(1-n)\bar k$. Since
\begin{equation*}
f(n+2,k-t)-f(n,k)=(n+1)(4a+\bar t)=0,
\end{equation*}
this map descends to $H$. Moreover,
\begin{align*}
f(n+n',k+k')-f(n,k)-f(n',k')
&=2nn'a-n'\bar k-n\bar k',\\
f(mn,mk)-m^2f(n,k)&=m(1-m)\bar k=0
\end{align*}
for $m,n,n'\in\Z$ and $k,k'\in K$. Thus the induced map $q_0:H\to E$ is homogeneous quadratic. Since $f(1,k)=a$, it is constant on $C_\chi$ and therefore induces a homomorphism $Q\to E$. Its values $q_0(x_0)=a$ and $q_0(k)=\bar k$ show that it is inverse to the displayed surjection.

The element $4a-\bar t$ has order $2$ in $(\Z/8\Z)a\oplus\bar K$. Hence
\begin{equation*}
[\Gamma(H):D_\chi]=|E|=\frac{8|K/2K|}{2}=2^{d(K)+2}.
\end{equation*}
Finally, $K[2]=\ker(\chi|_{H[2]})$ and $\dim_{\F_2}H[2]=r$, so
\begin{equation*}
d(K)=\dim_{\F_2}K[2]=
\begin{cases}
r-1,&\chi|_{H[2]}\ne0,\\
r,&\chi|_{H[2]}=0.
\end{cases}\qedhere
\end{equation*}
\end{proof}

The restriction map
$$
\Hom(H,\F_2)\rightarrow\Hom(H[2],\F_2)
$$
has rank $t(H)$. Indeed, on a cyclic summand of order $2^\lambda$, a functional restricts nontrivially to the order-$2$ subgroup exactly when $\lambda=1$. Its kernel therefore has dimension $r-t(H)$. Consequently the number of nonzero functionals with zero restriction to $H[2]$ is $2^{r-t(H)}-1$, while the number of functionals with nonzero restriction is $2^r-2^{r-t(H)}$. Lemma \ref{lem:Dchi-index} gives
\begin{equation}\label{eq:Dchi-sum}
\begin{aligned}
\sum_{0\ne\chi}|D_\chi|
&=|\Gamma(H)|\left(\frac{2^r-2^{r-t(H)}}{2^{r+1}}
+\frac{2^{r-t(H)}-1}{2^{r+2}}\right)\\
&=|\Gamma(H)|\left(\frac12-2^{-t(H)-2}-2^{-r-2}\right).
\end{aligned}
\end{equation}
Combining \eqref{eq:number-dyadic-orbits} and \eqref{eq:Dchi-sum} gives the stable number of dyadic orthogonal orbits
\begin{equation}\label{eq:dyadic-total-orbits}
T(H):=\#(\mO_n(\Z_2)\backslash\Sur(L_n,H))=|\Gamma(H)|\left(\frac32-2^{-t(H)-2}-2^{-r-2}\right).
\end{equation}

\subsection{Determinant-changing stabilizers}

In this subsection, we prove that the stabilizer of a surjection $f:L_n \twoheadrightarrow H$ in $\mO_n(\Z_2)$ contains an element of determinant $-1$ when $n > 2d(H)$.
For $x \in \F_2^n$, we write $\wt(x)$ for the Hamming weight of $x$, that is, the number of nonzero coordinates of $x$.

\begin{lem}
\label{lem:binary-code}
Let $C\subset\F_2^n$ be a subspace with $\dim_{\F_2} C>n/2$. Then the following holds.
\begin{enumerate}
\item There exists $x\in C$ such that $\wt(x)\not\equiv0\pmod4$.
\item If $C\subset\bar c_n^\perp$, then there exists $x\in C$ satisfying $\wt(x)\equiv2\pmod4$.
\end{enumerate}
\end{lem}

\begin{proof}
If every vector of $C$ had weight divisible by $4$, then
$$
\langle x,y\rangle
\equiv\frac{\wt(x)+\wt(y)-\wt(x+y)}2\equiv0\pmod2
$$
for all $x,y\in C$. Hence $C\subseteq C^\perp$, contradicting the assumption $2\dim_{\F_2} C>n$. This proves (a). Under the hypothesis of (b), every vector has even weight, so the vector supplied by (a) has weight congruent to $2$ modulo $4$.
\end{proof}

\begin{lem}
\label{lem:dyadic-stabilizer}
Let $f:L_n\twoheadrightarrow H$ be a surjection and assume $n>2d(H)$.
Then there exists $U\in\mO_n(\Z_2)$ such that
$$
f\circ U=f
\quad\text{and}\quad
\det(U)=-1.
$$
\end{lem}

\begin{proof}
Put $r=d(H)$ and $K=\ker f$. Let $\bar{f}:\bar {L}_n\to H/2H$ denote the reduction of $f$ modulo $2$. Then
$$
\bar{K}:=(K+2L_n)/2L_n=\ker(\bar{f}).
$$
Indeed, the inclusion $\bar{K}\subseteq\ker(\bar{f})$ is clear. Conversely, if
$f(x)\in 2H$, choose $y\in L_n$ such that $2f(y)=f(x)$. Then
$x-2y\in K$. Hence
$$
\dim_{\F_2}\bar K=n-r>n/2.
$$
By Lemma \ref{lem:binary-code}(a), choose
$\bar{v}\in\bar{K}$ such that
$$
\wt(\bar v)\not\equiv 0\pmod 4.
$$
Choose a lift $v\in K$ of $\bar v$. Since $\bar v\neq 0$, the vector
$v$ is primitive. Moreover,
$$
\langle v,v\rangle
\equiv \wt(\bar v) \not\equiv 0 \pmod 4.
$$
Thus $2/\langle v,v\rangle\in\Z_2$ and the reflection
$$
r_v(x)=x-2\frac{\langle x,v\rangle}{\langle v,v\rangle}v
$$
preserves $L_n$.

Over $\Q_2$, we have the orthogonal decomposition
$$
L_n\otimes\Q_2=v^\perp\oplus\Q_2v.
$$
The reflection $r_v$ acts as the identity on $v^\perp$ and as
multiplication by $-1$ on $\Q_2v$. Hence
$$
\det(r_v)=-1.
$$
Finally, since $v\in K=\ker f$, for every $x\in L_n$,
$$
f(r_v(x)-x)
=
-2\frac{\langle x,v\rangle}{\langle v,v\rangle}f(v)
=
0.
$$
Thus $r_v$ stabilizes $f$.
\end{proof}

\begin{rmk}
Suppose $n > 2d(H)$ and $n$ is sufficiently large for \eqref{eq:dyadic-orbit-bijection} to hold. By Lemma \ref{lem:dyadic-stabilizer}, every $\mO_n(\Z_2)$-orbit
in \eqref{eq:dyadic-orbit-bijection} is already a single
$\SO_n(\Z_2)$-orbit. Hence the same bijection holds with
$\mO_n(\Z_2)$ replaced by $\SO_n(\Z_2)$.
\end{rmk}


\section{The support of orthogonal cokernels at \texorpdfstring{$p=2$}{p=2}}\label{Sec8}

When $p=2$, a perfect skew-symmetric pairing is not necessarily alternating. The following elementary classification describes the support. 

\begin{prop}\label{prop:dyadic-support}
If a finite abelian $2$-group $M$ admits a perfect skew-symmetric pairing $\lambda: M\times M\to\Q_2/\Z_2$, then
$$
M\cong(\Z/2\Z)^a\oplus G\oplus G
$$
for some $a\in\{0,1\}$ and some finite abelian $2$-group $G$.
\end{prop}

\begin{proof}
Write
$$
M\cong\bigoplus_{e\ge1}(\Z/2^e\Z)^{m_e}.
$$
For $e\ge2$, let
$$
V_e=M[2^e]/(M[2^{e-1}]+2M[2^{e+1}]).
$$
Then $\dim_{\F_2}V_e=m_e$. 
Recall that a finite-dimensional vector space admitting a nondegenerate alternating bilinear form has even dimension. Therefore, it suffices to show that, for each $e\ge 2$, the vector space $V_e$ admits such a form.
The map
$$
(\bar x,\bar y)\mapsto
2^{e-1}\lambda(x,y)\in\tfrac12\Z/\Z\cong\F_2
$$
defines a well-defined pairing 
\begin{equation}
\label{eq:Ve-pairing}
V_e \times V_e \to \frac{1}{2}\Z/\Z.
\end{equation}
Indeed, if $a\in M[2^{e-1}]$, then
$$
2^{e-1}\lambda(a,y)=\lambda(2^{e-1}a,y)=0,
$$
while if $a=2z$ with $z\in M[2^{e+1}]$, then
$$
2^{e-1}\lambda(a,y)
=
2^e\lambda(z,y)
=
\lambda(z,2^ey)
=
0.
$$
The same argument applies to the second variable.
Moreover, skew-symmetry gives $2\lambda(x,x)= 0$, and for $e\ge 2$ this implies $2^{e-1}\lambda(x,x) = 0$ in $\Q_2/\Z_2$, so the pairing in \eqref{eq:Ve-pairing} is alternating. Now it remains to check nondegeneracy. Perfectness gives
\begin{equation}\label{eq:torsion-orthogonal-complement}
M[2^e]^\perp=2^eM.
\end{equation}
The inclusion $2^eM\subset M[2^e]^\perp$ is immediate, and equality follows because multiplication by $2^e$ has kernel $M[2^e]$, so both sides of \eqref{eq:torsion-orthogonal-complement} have order $|M|/|M[2^e]|$. If $x \in M[2^e]$ satisfies 
$$
2^{e-1}\lambda(x,y)=0
$$
for all $y\in M[2^e]$, then
$$
2^{e-1}x\in M[2^e]^\perp=2^eM.
$$
Thus $2^{e-1}x=2^ez$ for some $z\in M$, whence
$$
x-2z\in M[2^{e-1}] \quad\text{and}\quad 2^{e+1}z=0,
$$
so $\bar{x} = 0$ in $V_e$. This completes the proof.
\end{proof}

Let 
$$\delta= \begin{cases}
0 & \text{if $n\equiv 0 \pmod 2$}, \\
1 & \text{if $n\equiv 1 \pmod 2$}. 
\end{cases}
$$
Combining Proposition \ref{prop3a}, Proposition \ref{prop:dyadic-support}, and Theorem \ref{thm5a}, if $A_n$ is Haar-random in $\SO_n(\Z_2)$, then almost surely 
\begin{equation}\label{eq:dyadic-total-support}
\cok(A_n-I_n)\cong
\Z_2^\delta\oplus(\Z/2\Z)^a\oplus G\oplus G
\end{equation}
for some $a\in\{0,1\}$ and some finite abelian $2$-group $G$. Equivalently,
$$
\cok(A_n-I_n)\in
\mc S_{\sq}^{(\delta)}
\sqcup
\mc S_{\nsq}^{(\delta)}
$$
almost surely.


\section{Spinor parity and the dyadic parity-resolved moments}
\label{Sec9}
We first introduce an $\F_2$-valued invariant of $A\in\SO_n(\Z_2)$, namely the spinor character (Definition \ref{defn:spinor character}), and show that it detects whether the torsion subgroup of $\cok(A-I_n)$ is of the form $G\oplus G$ or $\Z/2\Z\oplus G\oplus G$. Next, for each surjection $f$, we determine whether this invariant is trivial on the stabilizer of $f$. Finally, we apply a weighted version of Burnside's lemma to compute separately the moments corresponding to the two cases. For simplicity, put $L = L_n = \Z_2^n$.


\subsection{A parity character on the orthogonal group} \label{Sub91}
Fix a characteristic vector $w\in L$ as in Section \ref{Sub64}:
\begin{equation}\label{eq:characteristic-vector}
\langle w,x \rangle \equiv \langle x,x \rangle\pmod2
\quad(x\in L).
\end{equation}
Every isometry preserves the unique characteristic class in $L/2L$, so $gw-w\in2L$ for every $g\in \mO(L)$. Define
\begin{equation}\label{eq:delta-w}
\delta_w(g)=\left\langle w,\frac{gw-w}{2}\right\rangle \pmod2.
\end{equation}

Put $V=L\otimes\Q_2$. For every $v\in V$ with $\langle v,v \rangle\ne0$, let
$$
r_v(x)=x-2\frac{\langle x,v \rangle}{\langle v,v\rangle}v
$$
be the reflection in $v$ as in the proof of Lemma \ref{lem:dyadic-stabilizer}. By the Cartan--Dieudonné theorem \cite[43:3 Theorem]{OMe73}, every $g\in\mO(V)$ can be written as a product $g=r_{v_1}\cdots r_{v_m}$. Define the \emph{spinor norm} by
$$
\mathrm{sp}_{-}:\mO(V)\rightarrow\Q_2^\times/(\Q_2^\times)^2,
\quad
\mathrm{sp}_{-}(g)=\prod_{i=1}^m (-\langle v_i,v_i \rangle)\cdot(\Q_2^\times)^2.
$$
The spinor norm $\mathrm{sp}_{-}$ is independent of the chosen reflection factorization \cite[\S55]{OMe73}. Compared with the usual spinor norm for $\langle \cdot, \cdot \rangle$, the displayed product has the extra factor $(-1)^m=\det(g)$; hence the two normalizations agree on $\SO(V)$. For $g \in \mO(V)$, put
$$
d(g)=\frac{1-\det g}{2}\in\{0,1\} = \F_2,
\quad
s(g)=v_2(\mathrm{sp}_{-}(g))\pmod2 \in \F_2.
$$
Both $d$ and $s$ are homomorphisms from $\mO(V)$ to $\F_2$.
\begin{lem}
\label{lem:characteristic-spinor}
For every $g\in \mO(L)$,
\begin{equation}\label{eq:characteristic-spinor}
\delta_w(g)=d(g)+s(g) \text{ in $\F_2$}.
\end{equation}
In particular, if $g\in \SO(L)$, then
\begin{equation*}
\delta_w(g)=v_2(\mathrm{sp}_{-}(g))\pmod2.
\end{equation*}
\end{lem}

\begin{proof}
First, note that $\delta_w$ is a homomorphism. If $z=(hw-w)/2\in L$, then
$$
\frac{ghw-w}{2}=gz+\frac{gw-w}{2}.
$$
Since $g^{-1}w\equiv w\pmod2$,
$$
\langle w,gz\rangle =\langle g^{-1}w,z\rangle \equiv \langle w,z\rangle\pmod2,
$$
and hence $\delta_w(gh)=\delta_w(g)+\delta_w(h)$.

The generation theorem of O'Meara and Pollak \cite[8.1]{OP65} states, in the present unimodular case, that $\mO(L)$ is generated by integral reflections and the Eichler transformations in \eqref{eq:eichler-transformation}. By the homomorphism properties, it suffices to verify \eqref{eq:characteristic-spinor} on these generators.

Let $r_v$ be an integral reflection. Scaling $v$ does not change $r_v$, so take $v\in L$ primitive (i.e., $v \notin 2L$). Unimodularity gives $\langle L,v \rangle =\Z_2$. Since $r_v(L)=L$, we have
$$
\frac{2}{\langle v,v \rangle}\in\Z_2
$$
so $e:=v_2(\langle v,v\rangle)\in\{0,1\}$.
Put $a=\langle w,v\rangle$. Equation \eqref{eq:characteristic-vector} gives $a\equiv \langle v,v\rangle\pmod2$, and
$$
\frac{r_vw-w}{2}=-\frac{a}{\langle v,v\rangle}v.
$$
Therefore
$$
\delta_w(r_v) \equiv -\frac{a^2}{\langle v,v\rangle}
\equiv
\begin{cases}
1& \text{ if $e=0$},\\
0& \text{ if $e=1$}.
\end{cases}
\pmod2.
$$
Since $d(r_v)=1$ and $s(r_v)=e$, identity \eqref{eq:characteristic-spinor} holds for integral reflections $r_v$. 

Now consider
\begin{equation}\label{eq:eichler-transformation}
E_{u,x}(z)=z+\langle z,u\rangle x-\langle z,x\rangle u
-\frac{\langle x,x\rangle }2\langle z,u\rangle u,
\end{equation}
where $u,x\in L$, $\langle u,u\rangle =\langle u,x\rangle =0$, and $\langle x,x\rangle \in2\Z_2$. The same formula is defined for $x\in u^\perp\subset V$. Direct substitution gives
\begin{equation*}
\langle E_{u,x}z,E_{u,x}z'\rangle =\langle z,z'\rangle ,\quad
E_{u,x}E_{u,y}=E_{u,x+y}
\end{equation*}
for $x,y\in u^\perp$ and $z,z'\in V$. In particular, $E_{u,x}^{-1}=E_{u,-x}$, and the integrality assumptions $u,x\in L$ and $\langle x,x\rangle \in2\Z_2$ imply that $E_{u,x}\in\mO(L)$. Over $\Q_2$, we have
\begin{equation*}
E_{u,x}=E_{u,x/2}^{\,2}\quad\text{in }\mO(V).
\end{equation*}
The targets of the determinant and the spinor norm have exponent $2$, so
\begin{equation*}
\det(E_{u,x})=1,\quad \mathrm{sp}_{-}(E_{u,x})=1.
\end{equation*}

Set $b=\langle w,u\rangle$. Since $\langle u,u\rangle =0$, equation \eqref{eq:characteristic-vector} gives $b\in2\Z_2$. Then \eqref{eq:eichler-transformation} yields
\begin{equation*}
\langle w,E_{u,x}w-w\rangle =-\frac{\langle x,x\rangle b^2}{2},\quad
\delta_w(E_{u,x})=-\frac{\langle x,x\rangle b^2}{4}\equiv0\pmod2.
\end{equation*}
Thus $\delta_w(E_{u,x})=d(E_{u,x})+s(E_{u,x})$. This proves \eqref{eq:characteristic-spinor} on all the generators and completes the proof.
\end{proof}

\begin{defn}\label{defn:spinor character}
For every $n\ge2$, define the continuous character 
$$
\eta:\SO_n(\Z_2)\rightarrow\F_2,
\quad
\eta(A)=v_2(\mathrm{sp}_{-}(A))\pmod2,
$$
which we call the \emph{spinor character}.
\end{defn}
Note that $\eta$ is surjective. Indeed, the integral reflections $r_{e_1}$ and $r_{e_1+e_2}$ have spinor norms represented by $-1$ and $-2$, respectively. Their product belongs to $\SO_n(\Z_2)$ and has spinor character value $1$. Consequently, each fiber of $\eta$ has Haar measure $1/2$.

We next connect the spinor character with the underlying cokernel. 

\begin{prop}
\label{prop:spinor-torsion-parity}
Let $n\ge2$. For every $A\in\SO_n(\Z_2)$,
\begin{equation}\label{eq:eta-torsion-parity}
\eta(A)=v_2(|\cok(A-I_n)_{\tors}|)\pmod2.
\end{equation}
If $n$ is even and $A_n$ is Haar-random in $\SO_n(\Z_2)$,
then almost surely $\cok(A_n-I_n)$ is finite and
\begin{equation}\label{eq:eta-det-parity}
\eta(A_n)=v_2(\det(A_n-I_n))\pmod2.
\end{equation}
\end{prop}

\begin{proof}
Let $A\in\SO_n(\Z_2)$, $T=A-I_n$, and $\Lambda=L\cap T(V)$. By \eqref{eq5a} and Theorem \ref{thm5a}, the group
\begin{equation*}
M=\cok(T)_{\tors}=\Lambda/TL
\end{equation*}
has a perfect skew-symmetric pairing $\lambda_A: M \times M \to \Q_2/\Z_2$ given by
\begin{equation*}
\lambda_A([x],[y])=\langle z,y\rangle +\Z_2,\quad Tz=x,
\end{equation*}
where $[x]$ denotes the class of $x\in\Lambda$ in $M$. Since $Tw\in2L$, we have $Tw/2\in\Lambda$, and
\begin{equation*}
\kappa=[Tw/2]\in M[2].
\end{equation*}
For $x=Tz\in\Lambda$, the isometry property of $A$ gives $2\langle z,Tz\rangle=-\langle Tz,Tz\rangle$. Together with \eqref{eq:characteristic-vector}, this implies that
\begin{equation*}
\lambda_A(\kappa,[x])
=\frac{\langle w,x\rangle }2+\Z_2
=-\frac{\langle x,x\rangle}2+\Z_2
=\lambda_A([x],[x]).
\end{equation*}
Moreover,
\begin{equation*}
\lambda_A(\kappa,\kappa)=\frac{\langle w,Tw\rangle}4+\Z_2.
\end{equation*}
By \eqref{eq:delta-w}, this value is $0$ when $\delta_w(A)=0$ and $\frac12+\Z_2$ when $\delta_w(A)=1$.

Set $C=\angles{\kappa}$, and let $C^\perp$ be its orthogonal complement in $M$ with respect to $\lambda_A$. The restriction of $\lambda_A$ to $C^\perp$ is alternating, since $\lambda_A(m,m)=\lambda_A(\kappa,m)=0$ for $m\in C^\perp$. Since $\lambda_A$ is perfect, $(C^\perp)^\perp=C$ and $[M:C^\perp]=|C|$.
If $\delta_w(A)=0$, then $C\subseteq C^\perp$, and the induced pairing on $C^\perp/C$ is perfect alternating. Hence $|M|=|C|^2|C^\perp/C|$ is a square. If $\delta_w(A)=1$, then $C$ has order $2$ and its restricted pairing is nondegenerate. Thus we have the orthogonal direct sum $M=C\perp C^\perp$, and the restricted pairing on $C^\perp$ is perfect alternating. Hence $|M|=2|C^\perp|$ is twice a square. Consequently,
\begin{equation*}
v_2(|M|)\equiv\delta_w(A)\pmod2.
\end{equation*}
Lemma \ref{lem:characteristic-spinor} gives $\delta_w(A)=\eta(A)$, proving \eqref{eq:eta-torsion-parity} for every $A\in\SO_n(\Z_2)$.

Now let $A_n$ be Haar-random in $\SO_n(\Z_2)$ and assume that
$n$ is even. Proposition \ref{prop3a} shows that
$\cok(A_n-I_n)$ is finite almost surely. On this event, Smith normal form gives
\begin{equation*}
v_2(|\cok(A_n-I_n)|)
=v_2(\det(A_n-I_n)),
\end{equation*}
so \eqref{eq:eta-det-parity} follows from \eqref{eq:eta-torsion-parity}.
\end{proof}


\subsection{The spinor image of a surjection stabilizer}\label{subsec:spinor-stabilizer}
Let $L= L_n$, $\bar{L}=\bar L_n$, $c_n$ and $\bar c_n$ be as in Section \ref{Sec6}. For a surjection $f:L_n\twoheadrightarrow H$, write $K=\ker f$. Recall that $f$ is characteristic if the functional $\bar x\mapsto\langle\bar c_n,\bar x\rangle = \langle \bar x, \bar x\rangle$ factors as
$$
\bar L_n\xrightarrow{\,\bar f\,}H/2H
\xrightarrow{\,\chi_f\,}\F_2
$$
for a nonzero functional $\chi_f:H\to\F_2$; after identifying $\Hom_{\F_2}(\bar L_n,\F_2)$ with $\bar L_n$ by the standard bilinear form, this is the identity $\bar c_n=\bar f^*\chi_f = \chi_f\bar f$.

We first isolate the vector needed to change spinor parity.

\begin{lem}
\label{lem:norm-two-kernel}
Let $r=d(H)$ and assume $n>2r+2$. Then $K$ contains a primitive vector $z$ satisfying
$$
\langle z,z\rangle\equiv2\pmod4.
$$
If $f$ is characteristic, the weaker bound $n>2r$ suffices.
\end{lem}

\begin{proof}
As in the proof of Lemma \ref{lem:dyadic-stabilizer}, put
$$
\bar K=(K+2L_n)/2L_n=\ker(\bar f).
$$
We have $\dim_{\F_2}(\bar{K})=n-r$.
If $f$ is characteristic and $n>2r$, then $\bar K\subset\bar c_n^\perp$.
Moreover,
$$
\dim_{\F_2}\bar K=n-r>\frac n2.
$$
Hence, by Lemma \ref{lem:binary-code}(b), $\bar K$ contains
a vector $\bar{z}$ of weight congruent to $2$ modulo $4$.

If $f$ is noncharacteristic and $n>2r+2$, let
$$
C=\bar K\cap\bar c_n^\perp.
$$
Since
$$
\dim_{\F_2}C\ge n-r-1>\frac n2,
$$
Lemma \ref{lem:binary-code}(b) gives a vector $\bar{z}$ in $C$ of
weight congruent to $2$ modulo $4$. 

In either case, a vector $\bar z$ of weight $2$ modulo $4$ lifts to a primitive $z\in K$ with
$$
\langle z,z\rangle \equiv\wt(\bar z)\equiv2\pmod4.
$$
\end{proof}

\begin{prop}
\label{prop:spinor-stabilizer-criterion}
Let $f:L_n\twoheadrightarrow H$, put $r=d(H)$, and assume $n>2r+2$. Then
\begin{equation}\label{eq:spinor-stabilizer-criterion}
\eta|_{\Stab_{\SO_n(\Z_2)}(f)}=0
\quad\Longleftrightarrow\quad
\begin{cases}
f\text{ is characteristic},\\
\chi_f|_{H[2]}=0.
\end{cases}
\end{equation}
Here $\Stab_{\SO_n(\Z_2)}(f)$ denotes the stabilizer of $f$ in $\SO_n(\Z_2)$.
\end{prop}

\begin{proof}
We have
$$
\bar K=(K+2L_n)/2L_n = \ker(\bar{f}).
$$
Hence finite-dimensional duality gives
$$
\im(\bar f^*)=\bar K^\perp.
$$

Suppose first that $f$ is noncharacteristic. Then $\bar c_n\notin\bar K^\perp$, so there is $\bar v\in\bar K$ with
$$
\langle\bar c_n,\bar v\rangle=1.
$$
Choose a lift $v\in K$. Since $\langle \bar{c}_n, \bar{v}\rangle = \langle \bar{v}, \bar{v} \rangle$, we have $\langle v,v\rangle \in\Z_2^\times$. Thus the reflection $r_v$ preserves $L_n$, stabilizes $f$, has determinant $-1$, and satisfies $s(r_v) = 0$. By Lemma \ref{lem:norm-two-kernel}, choose $z\in K$ with $v_2(\langle z,z\rangle )=1$. Then
$$
r_vr_z\in\Stab_{\SO_n(\Z_2)}(f)
\quad\text{and}\quad
\eta(r_vr_z) = 1.
$$

Now suppose that $f$ is characteristic and $\chi_f|_{H[2]}\ne0$. Choose $y\in H[2]$ with $\chi_f(y)=1$ and a lift $v\in L_n$ with $f(v)=y$. Since
$$
\langle v,v\rangle \equiv\chi_f(f(v))=1\pmod2,
$$
the reflection $r_v$ preserves $L_n$. It stabilizes $f$, because
$$
f(r_vx-x)=-2\frac{\langle x,v\rangle }{\langle v,v\rangle }y=0.
$$
Multiplying by the reflection $r_z$, where $z$ is given by
Lemma \ref{lem:norm-two-kernel}, gives
$$
r_vr_z\in\Stab_{\SO_n(\Z_2)}(f)
\quad \text{and} \quad
\eta(r_vr_z)=1.
$$
Thus $\eta|_{\mathrm{Stab}_{\SO_n(\Z_2)}(f)}$ is nontrivial in this case.

Finally suppose that $f$ is characteristic and $\chi_f|_{H[2]}=0$. Write
$$
H=\bigoplus_i\Z/2^{\lambda_i}\Z\,h_i.
$$
Define a homomorphism
$$
\widetilde\chi_f:H\to\Z/4\Z
$$
by setting $\widetilde\chi_f(h_i)=0$ when $\lambda_i=1$, and, when
$\lambda_i\ge2$, setting $\widetilde\chi_f(h_i)=0$ or $1$ according as
$\chi_f(h_i)=0$ or $1$. This defines a homomorphism because
$4\mid 2^{\lambda_i}$ whenever $\lambda_i\ge2$. Its reduction modulo $2$
is $\chi_f$: this is clear when $\lambda_i\ge2$, while for
$\lambda_i=1$ it follows from the assumption $\chi_f|_{H[2]}=0$.

By unimodularity there is $u\in L_n$, unique modulo $4L_n$, such that
$$
\langle u,x\rangle \equiv\widetilde\chi_f(f(x))\pmod4
\quad(x\in L_n).
$$
We have $\bar{u} = \bar{c}_n$ in $L_n/2L_n$. 
If $U\in\Stab_{\mO_n(\Z_2)}(f)$, then
$$
\langle U^{-1}u-u,x\rangle 
=\langle u,Ux\rangle -\langle u,x\rangle \equiv0\pmod4
$$
for every $x\in L_n$. Unimodularity gives
$$
U^{-1}u\equiv u\pmod{4L_n},
$$
and hence also $Uu\equiv u\pmod{4L_n}$. For $U\in\SO_n(\Z_2)$ stabilizing $f$, equation \eqref{eq:delta-w} gives $\delta_u(U)=0$. Lemma \ref{lem:characteristic-spinor} then gives $\eta(U)=0$. This proves \eqref{eq:spinor-stabilizer-criterion}.
\end{proof}


\subsection{Character-weighted Burnside and the two moment families}\label{subsec:twisted-burnside}

For a finite abelian $2$-group $H$, put $r=d(H)$. Let $A_n$ be Haar-random in $\SO_n(\Z_2)$ and define
\begin{align*}
\M_{0,n}(H)&=\bE(\mathbf{1}_{\{\eta(A_n)=0\}}\#\Sur(\cok(A_n-I_n),H)),\\
\M_{1,n}(H)&=\bE(\mathbf{1}_{\{\eta(A_n)=1\}}\#\Sur(\cok(A_n-I_n),H)).
\end{align*}
We assume that $n>2r+2$ and that $n$ is sufficiently large for \eqref{eq:dyadic-orbit-bijection} to hold. By Proposition \ref{prop:spinor-stabilizer-criterion}, the character $(-1)^\eta$ is trivial on a stabilizer exactly for the characteristic orbits with $\chi|_{H[2]}=0$. Their number is (see Section \ref{Sub72})
\begin{equation}\label{eq:special-orbits}
S(H)
:=\sum_{\substack{0\ne\chi\\\chi|_{H[2]}=0}}|D_\chi|
=(2^{r-t(H)}-1)\frac{|\Gamma(H)|}{2^{r+2}}
=|\Gamma(H)|(2^{-t(H)-2}-2^{-r-2}).
\end{equation}

On any $\SO_n(\Z_2)$-orbit $\Omega$ in $\Sur(L_n, H)$, Lemma \ref{lem4a} gives ordinary fixed-point average $1$. Its $(-1)^\eta$-weighted average is $1$ on the $S(H)$ special orbits and $0$ on every other orbit. Since
$$
\mathbf{1}_{\{\eta=0\}}=\frac{1+(-1)^\eta}{2},
\quad
\mathbf{1}_{\{\eta=1\}}=\frac{1-(-1)^\eta}{2},
$$
we obtain directly from \eqref{eq:dyadic-total-orbits} and \eqref{eq:special-orbits} (cf. \eqref{eq:cf1} and \eqref{eq:cf2})
\begin{equation}\label{eq:parity-M} 
\begin{aligned}
\M_{0,n}(H)
&=\frac{T(H)+S(H)}2
=|\Gamma(H)|\left(\frac34-2^{-r-2}\right),\\
\M_{1,n}(H)
&=\frac{T(H)-S(H)}2
=|\Gamma(H)|\left(\frac34-2^{-t(H)-2}\right).
\end{aligned}
\end{equation}
Taking $H=0$, we obtain
$$
\bP(\eta(A_n)=0)=\bP(\eta(A_n)=1)=\frac{1}{2}.
$$
By Proposition \ref{prop:spinor-torsion-parity}, for $A \in \SO_n(\Z_2)$
$$
\eta(A) = v_2(|\cok(A-I_n)_{\tors}|)\pmod2.
$$
Therefore, dividing the two identities in
\eqref{eq:parity-M} by $1/2$ proves Theorem \ref{thm1c}.

Adding the two identities in \eqref{eq:parity-M} gives
$$
\bE(\#\Sur(\cok(A_n-I_n),H))
=|\Gamma(H)|\left(\frac32-2^{-t(H)-2}-2^{-r-2}\right).
$$

For an $\mO_n(\Z_2)$-orbit $\Omega$, Lemma \ref{lem:dyadic-stabilizer}
shows that the determinant character is nontrivial on its stabilizer.
Hence, applying Lemma \ref{lem4a} to the full group
$\mO_n(\Z_2)$, we have
$$
\int_{\mO_n(\Z_2)} \det(A)\#\Fix_\Omega(A)\,dA=0.
$$
On the other hand, the same lemma with the trivial character gives
$$
\int_{\mO_n(\Z_2)} \#\Fix_\Omega(A)\,dA=1.
$$
For $\tau\in\{\pm1\}$, let
$$
\mO_n(\Z_2)^\tau=\{A\in\mO_n(\Z_2):\det(A)=\tau\}.
$$
Since
$$
\mathbf{1}_{\mO_n(\Z_2)^\tau}(A)
=\frac{1+\tau\det(A)}2,
$$
we obtain 
$$
\int_{\mO_n(\Z_2)^\tau} \#\Fix_\Omega(A)\,dA 
= \int_{\mO_n(\Z_2)} \mathbf{1}_{\mO_n(\Z_2)^\tau}(A) \#\Fix_\Omega(A)\,dA 
= \frac12.
$$
As $\mO_n(\Z_2)^\tau$ has Haar measure $1/2$, its normalized
fixed-point average is $1$. Then summing over the $T(H)$ orbits proves Theorem \ref{thm1a}(b).

Finally, returning to the Haar-random matrix $A_n\in\SO_n(\Z_2)$, let $\delta\in\{0,1\}$ be determined by $\delta\equiv n\pmod2$. Proposition \ref{prop:spinor-torsion-parity} identifies the two supports.
By \eqref{eq:dyadic-total-support}, almost surely
$$
\cok(A_n-I_n)\cong
\Z_2^\delta\oplus(\Z/2\Z)^a\oplus G\oplus G
$$
for some $a\in\{0,1\}$ and some finite abelian $2$-group $G$. Its torsion subgroup has order $2^a|G|^2$, and hence almost surely 
\begin{align*}
\eta(A_n)=0\quad &\Longleftrightarrow\quad
\cok(A_n-I_n)\in\mc S_{\sq}^{(\delta)},\\
\eta(A_n)=1\quad &\Longleftrightarrow\quad
\cok(A_n-I_n)\in\mc S_{\nsq}^{(\delta)}.
\end{align*}

The next section proves moment determinacy on these two supports and identifies the limiting laws.


\section{\texorpdfstring{Moment inversion on the base-$4$ grid}{Moment inversion on the base-4 grid}}\label{Sec10}

We adapt Wood's multivariable interpolation argument from
\cite[Lemma 8.1 and Theorem 8.2]{Woo17}. Although these results are
stated for primes, their proofs apply verbatim with any real number
$q>1$ in place of the prime. We only need the case $q=4$. For $m\ge 1$, let $\mc P_m$ be the set of partitions
$$
\lambda=(\lambda_1\ge\cdots\ge\lambda_m\ge0),
$$
and for $\lambda, \mu \in \mc{P}_m$, define
$$
\lambda\cdot\mu=\sum_{i=1}^m\lambda_i\mu_i.
$$

\begin{thm}[{\cite[Theorem 8.2]{Woo17}}]
\label{thm:base4-determinacy}
Let $x_\mu,y_\mu\ge0$ for $\mu\in\mc P_m$. Suppose that, for every $\lambda\in\mc P_m$,
\begin{equation*}
\sum_\mu x_\mu4^{\lambda\cdot\mu}
=\sum_\mu y_\mu4^{\lambda\cdot\mu}
=C_\lambda,
\end{equation*}
and that
\begin{equation*}
C_\lambda\le F^m4^{\frac12\sum_i\lambda_i(\lambda_i-1)}
\end{equation*}
for some $F>0$. Then $x_\mu=y_\mu$ for every $\mu \in \mc P_m$.
\end{thm}

\subsection{Application to the dyadic support}\label{subsec:base-four}

For $n\ge2$, let $A_n$ be Haar-random in $\SO_n(\Z_2)$.
For $b\in\{0,1\}$, let $C_{n,b}$ have the conditional distribution of
$\cok(A_n-I_n)$ given $\eta(A_n)=b$, and put
$$
T_{n,b}:=(C_{n,b})_{\tors}.
$$
By Proposition \ref{prop:spinor-torsion-parity},
$\eta(A_n)=\ve(A_n)$, so conditioning on $\eta(A_n)=b$ is equivalent
to conditioning on $\ve(A_n)=b$. For $n\equiv\delta\pmod2$ with
$\delta\in\{0,1\}$, Proposition \ref{prop3a} gives
$$
C_{n,b}\cong\Z_2^\delta\oplus T_{n,b}
\quad\text{(almost surely).}
$$

For a finite abelian $2$-group $M$, let
$$
\rho_i(M)=\dim_{\F_2}(2^{i-1}M/2^iM)
\quad(i\ge1).
$$
Thus $\rho(M)=(\rho_1(M)\ge\rho_2(M)\ge\cdots)$ is a partition. Let $\lambda = \rho(H)$. Then
$$
\#\Hom(M,H)=2^{\sum_i\lambda_i\rho_i(M)}.
$$
If $M=G\oplus G$ and $\nu=\rho(G)$, then
\begin{equation}\label{eq:square-base4}
\#\Hom(M,H)=4^{\lambda\cdot\nu}.
\end{equation}
If $M=\Z/2\Z\oplus G\oplus G$, then
\begin{equation}\label{eq:nonsquare-base4}
\#\Hom(M,H)=2^{\lambda_1}4^{\lambda\cdot\nu}.
\end{equation}
For $X=\Z_2^\delta\oplus T$,
$$
\#\Hom(X,H)=|H|^\delta\#\Hom(T,H).
$$
Thus, after division by the known factors $|H|^\delta$ and, in the nonsquare case, $2^{\lambda_1}$, both conditional moment problems are base-$4$ moment problems.

We record the required growth bound. For each fixed finite abelian
$2$-group $J$ and all sufficiently large $n$, the conditional formulas
\eqref{eq1c} give
\begin{equation} \label{eq:Sur-moment-inequality}
\bE(\#\Sur(C_{n,b},J))\le\frac32|\Gamma(J)|. 
\end{equation}
Moreover, \eqref{eq:gamma-decomposition} implies that
$$
|\Gamma(J)|=2^{d(J)}|J|\,|\wedge^2 J|,
$$
where $\wedge^2 J$ denotes the second exterior power of $J$ over $\Z$. If $H$ has exponent dividing $2^m$ and $\lambda=\rho(H)$, then for $n\equiv\delta\pmod2$ with $\delta \in \{0,1\}$,
$$
\bE(\#\Hom(T_{n,b},H))
=|H|^{-\delta}\sum_{J\le H}\bE(\#\Sur(C_{n,b},J)).
$$
Wood's subgroup estimate \cite[Lemma 7.5]{Woo17} gives\footnote{Wood \cite{Woo17} says that $K$ is of type $\mu=(\mu_1,\ldots,\mu_r)$ if
$K\cong\bigoplus_{j=1}^r\Z/2^{\mu_j}\Z$. Thus, in our notation, $\rho(K) = \mu'$, where $\mu'$ denotes the conjugate partition of $\mu$.}
$$
\sum_{J\le H}|\wedge^2 J|
\le F_0^m2^{\frac12\sum_i\lambda_i(\lambda_i-1)}
$$
for a constant $F_0>0$. Since $d(J)\le\lambda_1$ and $|J|\le2^{\sum_i\lambda_i}$,
for all sufficiently large $n$ (depending on $H$),
$$
\bE(\#\Hom(T_{n,b},H))
\le\frac32F_0^m
2^{\lambda_1+\sum_i\lambda_i+\frac12\sum_i\lambda_i(\lambda_i-1)}.
$$
We have
$$
\lambda_1 + \sum_i\lambda_i \le 2\sum_i\lambda_i
\le\frac12\sum_i\lambda_i(\lambda_i-1)+3m,
$$
by $2k\le k(k-1)/2+3$ for $k\ge1$. Absorbing the factor $2^{3m}$ into the constant gives
\begin{equation}\label{eq:base4-growth}
\bE(\#\Hom(T_{n,b},H))
\le F_1^m4^{\frac12\sum_i\lambda_i(\lambda_i-1)}.
\end{equation}
Note that the same bound applies to the torsion part of any probability law on $\mc S_{\sq}^{(\delta)}$ or $\mc S_{\nsq}^{(\delta)}$ whose surjection moments are given by the corresponding right-hand side of \eqref{eq1c} for every finite target.

\begin{thm}
\label{thm:dyadic-conditional-determinacy}
Fix $\delta\in\{0,1\}$. As $n\to\infty$ through integers satisfying
$n\equiv\delta\pmod2$, the laws
$$
\mc L(C_{n,0}) \quad\text{and}\quad \mc L(C_{n,1})
$$
converge to probability laws supported on $\mc S_{\sq}^{(\delta)}$ and $\mc S_{\nsq}^{(\delta)}$, respectively. Moreover, each limiting law is uniquely determined by its $H$-moments for all finite abelian $2$-groups $H$: it is the unique probability law on the corresponding support whose $H$-moments are given by the corresponding right-hand side of \eqref{eq1c} for every finite abelian $2$-group $H$.
\end{thm}

\begin{proof}
We use the random $\Z_2$-modules $C_{n,b}$ and $T_{n,b}$ defined above, with $b\in\{0,1\}$ and $n\equiv\delta\pmod2$. We first prove tightness. Put $E_k=(\Z/2\Z)^k$. If $d(T_{n,b})\ge k$, then $C_{n,b}$ surjects onto $E_k$, and postcomposition gives at least $|\GL_k(\F_2)|$ surjections. Hence
$$
|\GL_k(\F_2)|\, \mathbf{1}_{\{d(T_{n,b})\ge k\}}
\le \#\Sur(C_{n,b},E_k).
$$
Since $|\Gamma(E_k)|=2^{k(k+3)/2}$, it follows from \eqref{eq:Sur-moment-inequality} that, for all sufficiently large $n$,
\begin{equation*}
\bP(d(T_{n,b})\ge k) \le \frac{\bE(\#\Sur(C_{n,b},E_k))}{|\GL_k(\F_2)|}
\le\frac{(3/2)|\Gamma(E_k)|}{|\GL_k(\F_2)|}
\ll2^{-(k^2-3k)/2}.
\end{equation*}
Let $C_m=(\Z/2^m\Z)^2$ with $m\ge2$. Let $\exp(T_{n,b})$ denote the exponent of $T_{n,b}$. On either support, if $\exp(T_{n,b})\ge2^m$, then by Theorem \ref{thm5a} and Proposition \ref{prop:dyadic-support}, there is a surjection $T_{n,b}\twoheadrightarrow C_m$. Since
$$
|\Gamma(C_m)|=2^{3m+2}
\quad\text{and}\quad
|\Aut(C_m)|=3\cdot2^{4m-3},
$$
it follows from \eqref{eq:Sur-moment-inequality} similarly as above that, for all sufficiently large $n$,
\begin{equation*}
\bP(\exp(T_{n,b})\ge2^m)
\le\frac{(3/2)|\Gamma(C_m)|}{|\Aut(C_m)|}
=16\cdot 2^{-m}.
\end{equation*}
Only finitely many finite abelian $2$-groups have bounded rank and exponent, so both conditional sequences are tight.

Fix $b\in\{0,1\}$ and take a convergent subsequence of $T_{n,b}$, still indexed by $n$. Let $T_b$ denote its limit in distribution. Since the state space of finite abelian $2$-groups is countable and discrete, and each of the square and nonsquare support classes is closed, the law of $T_b$ is supported on the same class as the laws of $T_{n,b}$.

We next show that $T_b$ inherits the prescribed moments. For $n\equiv\delta\pmod 2$, we have almost surely
$$
C_{n,b} \cong \Z_2^\delta\oplus T_{n,b}.
$$
Therefore, for every finite abelian $2$-group $H$, we have almost surely
$$
\#\Hom(C_{n,b},H)=|H|^\delta\,\#\Hom(T_{n,b},H).
$$

Also, for every finitely generated $\Z_2$-module $X$
$$
\#\Hom(X,H)=\sum_{J\le H}\#\Sur(X,J).
$$
Hence
$$
\bE(\#\Hom(T_{n,b},H))=|H|^{-\delta}
\sum_{J\le H}\bE(\#\Sur(C_{n,b},J)),
$$
and the right-hand side converges by \eqref{eq1c}.

To pass these moments to the limit, note that
$$
\#\Hom(T_{n,b},H)^2=\#\Hom(T_{n,b},H\oplus H).
$$
Applying the same formula with $H\oplus H$ in place of $H$ gives
$$
\sup_n \bE(\#\Hom(T_{n,b},H)^2)<\infty.
$$
Hence the random variables $\#\Hom(T_{n,b},H)$ are uniformly integrable. Since the state space is discrete and $T_{n,b}$ converges in distribution
to $T_b$, we have
$$
\#\Hom(T_{n,b},H)\xrightarrow{d}\#\Hom(T_b,H)
$$
in distribution. Uniform integrability therefore gives
$$
\bE(\#\Hom(T_{n,b},H)) \rightarrow \bE(\#\Hom(T_b,H)).
$$
Thus $T_b$ has the prescribed Hom moments and satisfies the growth
bound \eqref{eq:base4-growth} with $T_b$ in place of $T_{n,b}$.

Fix $m\ge1$ and put $T_b^{[m]}=T_b/2^mT_b$. On the square support there is a finite group $G_m$, unique up to isomorphism, such that
$$
T_b^{[m]}=G_m\oplus G_m.
$$
On the nonsquare support,
$$
T_b^{[m]}=\Z/2\Z\oplus G_m\oplus G_m.
$$
For targets killed by $2^m$, the Hom moments of $T_b$ are exactly those of $T_b^{[m]}$. Equations \eqref{eq:square-base4} and \eqref{eq:nonsquare-base4} therefore give all base-$4$ moments of the partition $\rho(G_m)\in\mc P_m$. They satisfy \eqref{eq:base4-growth}, so Theorem \ref{thm:base4-determinacy} makes the law of every truncation $T_b^{[m]}$ unique.

Finally, fix a finite abelian $2$-group $M$ and choose $m$ with $2^m>\exp(M)$. For every finite abelian $2$-group $T$,
$$
T/2^mT\cong M
\quad\Longleftrightarrow\quad
T\cong M.
$$
Thus the truncation laws determine every point mass of $T_b$. Any two subsequential limits are equal, so the laws of $T_{n,b}$ converge. Since $C_{n,b}\cong\Z_2^\delta\oplus T_{n,b}$ almost surely, the laws of $C_{n,b}$ also converge. The same argument for two arbitrary laws with the prescribed moments proves uniqueness. 
\end{proof}

\subsection{Explicit point masses of the dyadic limits}
\label{subsec:dyadic-masses}
In this section, we prove Theorem \ref{thm1d}. Put $E=\Z/2\Z$ and, for $\delta\in\{0,1\}$, put
$$
c_{2,\delta}=\prod_{i=\delta}^{\infty}(1-2^{-(2i+1)}).
$$
The explicit even- and odd-dimensional alternating laws of Bhargava, Kane, Lenstra, Poonen and Rains \cite[Theorems 3.9 and 3.11]{BKLPR15} give probability laws $\nu_{\mathrm{alt}}^{(\delta)}$ on $\mc S_{\mathrm{sq}}^{(\delta)}$ with
\begin{align*}
\nu_{\mathrm{alt}}^{(0)}(M)
&=c_{2,0}\frac{|M|}{|\Sp(M)|},\\
\nu_{\mathrm{alt}}^{(1)}(\Z_2\oplus M)
&=\frac{c_{2,1}}{|\Sp(M)|},
\end{align*}
for every finite symplectic $2$-group $M$, with mass zero on every other $\Z_2$-module. Here, $\Sp(M)$ denotes the group of automorphisms of $M$ preserving a fixed nondegenerate alternating pairing.
For both values of $\delta$, their surjection moments are
\begin{equation}\label{eq:alt-parity-sur-moment}
\int\#\Sur(X,H)\,d\nu_{\mathrm{alt}}^{(\delta)}(X)
=|\Sym^2H|.
\end{equation}
To see this, let
$$
Z_g=\#\Sur(\cok(B_{2g+\delta}),H)
$$
for additive Haar-random $B_{2g+\delta}\in\Alt_{2g+\delta}(\Z_2)$. Then \cite[Theorem 3.1]{NW25} and the argument in the proof of \cite[Theorem 1.13]{NW25} give
$$
\lim_{g\to\infty}\bE(Z_g) = |\Sym^2H|.
$$
Moreover,
$$
Z_g^2\le
\#\Hom(\cok(B_{2g+\delta}),H\oplus H)
=\sum_{J\le H\oplus H} \#\Sur(\cok(B_{2g+\delta}),J),
$$
and the expectation of the right hand side is bounded uniformly in $g$ by \cite[Theorem 3.1]{NW25}. 
Thus $(Z_g)$ is uniformly integrable. By \cite[Theorems 3.9 and 3.11]{BKLPR15}, $\cok(B_{2g+\delta})$ converges in distribution to $\nu_{\mathrm{alt}}^{(\delta)}$. Combining this convergence with uniform integrability, we obtain
$$
\int \#\Sur(X,H)\,d\nu_{\mathrm{alt}}^{(\delta)}(X)
=\lim_{g\to\infty}\bE(Z_g)
=|\Sym^2H|.
$$

Since $\nu_{\mathrm{alt}}^{(\delta)}(M_\delta)=w_\delta(M)$, the definitions of $\mu_{\mathrm{sq}}^{(\delta)}$ and $\mu_{\mathrm{nsq}}^{(\delta)}$ give
\begin{equation}\label{eq:parity-square-size-biased-law}
\mu_{\mathrm{sq}}^{(\delta)}(X)
=\frac{\#\Sur(X,E)}2 \, \nu_{\mathrm{alt}}^{(\delta)}(X)
\end{equation}
and
$$
\mu_{\mathrm{nsq}}^{(\delta)}(Y)=
\begin{cases}
\nu_{\mathrm{alt}}^{(\delta)}(X) & \text{if $Y\cong E\oplus X$ for some $X\in\mathcal S_{\mathrm{sq}}^{(\delta)}$,}\\
0 & \text{otherwise.}
\end{cases}
$$
The second measure is obviously a probability law. The first is also a probability law, since \eqref{eq:alt-parity-sur-moment} with target $E$ gives
$$
\int\frac{\#\Sur(X,E)}2\,
d\nu_{\mathrm{alt}}^{(\delta)}(X)
=\frac{|\Sym^2E|}{2}=1.
$$

We first compute the square moments. 
For every finitely generated $\Z_2$-module $X$
and every finite abelian $2$-group $H$,
\begin{equation}\label{eq:paired-surjection-identity}
\#\Sur(X,E)\#\Sur(X,H)
=\#\Sur(X,E\oplus H)
+\#\Sur(H,E)\#\Sur(X,H).
\end{equation}
Indeed, let $(a,b):X\to E\oplus H$ be defined by two surjections $a$ and $b$. If $(a,b)$ is not surjective, its image still projects surjectively to $H$ and has order at least $|H|$. Since it is a proper subgroup of $E\oplus H$, it has order $|H|$, and the projection to $H$ is an isomorphism. Hence the image is the graph of a unique surjection $H\twoheadrightarrow E$. Conversely every such graph gives a nonsurjective pair. Put $r=d(H)$. Since
$$
\#\Sur(H,E)=2^r-1
$$
and
$$
\Sym^2(E\oplus H) \cong\Sym^2E\oplus(E\otimes H)\oplus\Sym^2H,
$$
we have
$$
|\Sym^2(E\oplus H)|=2^{r+1}|\Sym^2H|.
$$
Equations \eqref{eq:alt-parity-sur-moment}, \eqref{eq:parity-square-size-biased-law}, and \eqref{eq:paired-surjection-identity} therefore give, for either $\delta$,
\begin{align*}
\int\#\Sur(X,H)\,d\mu_{\mathrm{sq}}^{(\delta)}(X)
&=\frac12 (2^{r+1}+2^r-1)|\Sym^2H|\\
&=|\Gamma(H)|\left(\frac32-2^{-r-1}\right),
\end{align*}
where the last equality uses \eqref{eq:gamma-size}.

We next compute the nonsquare moments. A homomorphism $E\oplus X\to H$ is determined by the image $h\in H[2]$ of the generator of $E$ and by a homomorphism $X\to H$. If $J$ is the image of the latter, then the combined map is surjective exactly when $H=J+\angles{h}$. If $J=H$, the expected contribution is $2^r|\Sym^2H|$. Otherwise $J$ has index $2$ and equals $\ker\chi$ for a nonzero functional $\chi:H\to E$. Such an $h$ exists exactly when $\chi|_{H[2]}\ne0$, and then there are $2^{r-1}$ choices for $h$. The restriction map
$$
\Hom(H,E)\rightarrow\Hom(H[2],E)
$$
has rank $t(H)$, so the number of these functionals is $2^r-2^{r-t(H)}$. For such $\chi$, the affine hyperplane
$$
\{h\in H[2]:\chi(h)=1\}
$$
has $2^{r-1}$ elements. Choose one of them. Then
$$
H=J\oplus\angles{h}\cong J\oplus E,
\quad d(J)=r-1.
$$
Using
$$
\Sym^2(J\oplus E)
\cong\Sym^2J\oplus(J\otimes E)\oplus\Sym^2E,
$$
we obtain
$$
|\Sym^2H|=2^{d(J)+1}|\Sym^2J|=2^r|\Sym^2J|.
$$
Using \eqref{eq:alt-parity-sur-moment}, we obtain, again for either $\delta$,
\begin{align*}
\int\#\Sur(Y,H)\,d\mu_{\mathrm{nsq}}^{(\delta)}(Y)
&=2^r|\Sym^2H|
+(2^r-2^{r-t(H)})2^{r-1}2^{-r}|\Sym^2H|\\
&=|\Gamma(H)|\left(\frac32-2^{-t(H)-1}\right).
\end{align*}
Thus the two measures have the moments in \eqref{eq1c}. 
By Theorem \ref{thm:dyadic-conditional-determinacy}, the limiting laws of $C_{n,0}$ and $C_{n,1}$ are $\mu_{\mathrm{sq}}^{(\delta)}$ and
$\mu_{\mathrm{nsq}}^{(\delta)}$, respectively. Thus Theorem \ref{thm1d}(b) follows.
Since each fiber of the spinor character has Haar measure $1/2$, Theorem \ref{thm1d}(c) follows as well. Parts (a) and (d) are proved above, completing the proof of Theorem \ref{thm1d}.

\section*{Acknowledgments}

Jungin Lee was supported by the National Research Foundation of Korea (NRF) grant funded by the Korea government (MSIT) (No. RS-2024-00334558 and No. RS-2025-02262988). Myungjun Yu was supported by the National Research Foundation of Korea (NRF) grant funded by the Korea government (MSIT) (No. RS-2025-23525445).

\section*{Statement on AI use}

The authors conceived the project of studying the linearization of the random $p$-adic orthogonal matrix model, formulated the central research questions, and established the overall research direction. They proposed determining the limiting cokernel distributions through moment computations as the basic strategy.

Building on the research questions and strategy proposed by the authors, OpenAI's ChatGPT 5.6 and 6 Pro developed the specific mathematical arguments and generated proofs of the main results through an iterative dialogue with the authors. In particular, for $p=2$, ChatGPT identified that linearization does not preserve the limiting cokernel distribution and proposed using the spinor norm to distinguish the two parity cases and determine the corresponding limiting distributions.

The authors critically evaluated and revised the model's suggestions, independently verified all mathematical arguments, and wrote the final manuscript. They take full responsibility for the correctness of the results and the content of the paper.


\end{document}